\pdfoutput=1
\documentclass[11pt]{amsart}
\usepackage{amsmath,amssymb,color,amsthm,amsfonts}
\usepackage{mathtools}
\usepackage[style=alphabetic]{biblatex}
\usepackage{tikz-cd}

\usepackage[mathscr]{euscript}
\usepackage{fullpage}
\usepackage{hyperref}

\newcommand{\on}{\operatorname}

\DeclareMathOperator{\crit}{crit}

\DeclareMathOperator{\QCoh}{QCoh}
\DeclareMathOperator{\Op}{Op}

\DeclareMathOperator{\Whit}{Whit}

\DeclareMathOperator{\Gr}{Gr}
\DeclareMathOperator{\Sym}{Sym}

\DeclareMathOperator{\Hom}{Hom}
\DeclareMathOperator{\sHom}{\sH om}
\DeclareMathOperator{\Spec}{Spec}

\DeclareMathOperator{\gen}{-gen}
\DeclareMathOperator{\IndCoh}{IndCoh}
\DeclareMathOperator{\Hecke}{Hecke}
\DeclareMathOperator{\Heckenaive}{Hecke,naive}
\DeclareMathOperator{\reg}{reg}
\DeclareMathOperator{\regnaive}{reg,naive}
\newcommand{\xar}[1]{\xrightarrow{#1}}
\newcommand{\DGCat}{\mathsf{DGCat}}

\renewcommand{\mod}{\text{\textendash}\on{mod}}

\newcommand{\bD}{{\mathbb D}}

\newcommand{\sA}{{\EuScript A}}
\newcommand{\sB}{{\EuScript B}}
\newcommand{\sC}{{\EuScript C}}
\newcommand{\sD}{{\EuScript D}}
\newcommand{\sE}{{\EuScript E}}
\newcommand{\sF}{{\EuScript F}}
\newcommand{\sG}{{\EuScript G}}
\newcommand{\sH}{{\EuScript H}}

\newcommand{\sM}{{\EuScript M}}

\newcommand{\sO}{{\EuScript O}}
\newcommand{\sP}{{\EuScript P}}

\newcommand{\sS}{{\EuScript S}}

\newcommand{\fZ}{{\mathfrak Z}}

\newcommand{\fg}{{\mathfrak g}}
\newcommand{\fh}{{\mathfrak h}}

\newcommand{\fk}{{\mathfrak k}}

\newcommand{\fn}{{\mathfrak n}}

\newcommand{\ft}{{\mathfrak t}}

\newcommand{\fz}{{\mathfrak z}}

\newcommand\rightthreearrow{%
        \mathrel{\vcenter{\mathsurround0pt
                \ialign{##\crcr
                        \noalign{\nointerlineskip}$\rightarrow$\crcr
                        \noalign{\nointerlineskip}$\rightarrow$\crcr
                        \noalign{\nointerlineskip}$\rightarrow$\crcr
                }%
        }}%
}

\newtheorem{theorem}{Theorem}[section]
\newtheorem{lemma}[theorem]{Lemma}
\newtheorem{proposition}[theorem]{Proposition}
\newtheorem{corollary}[theorem]{Corollary}
\newtheorem{definition}[theorem]{Definition}

\theoremstyle{remark}
\newtheorem{example}[theorem]{Example}
\newtheorem{remark}[theorem]{Remark}

\title{Affine Beilinson-Bernstein localization at the critical level}

\author{Sam Raskin}

\address{The University of Texas at Austin,
Department of Mathematics,
PMA 8.100, 2515 Speedway Stop C1200,
Austin, TX 78712}

\email{sraskin@math.utexas.edu}

\author{David Yang}

\address{Harvard University,
Department of Mathematics,
1 Oxford St,
Cambridge, MA 02138}

\email{dyang@math.harvard.edu}

\begin{document}

\maketitle

\begin{abstract}

We prove the Frenkel-Gaitsgory localization conjecture describing regular
Kac-Moody representations at critical level via eigensheaves
on the affine Grassmannian using categorical Moy-Prasad theory. 
This extends previous work of the authors.

\end{abstract}

\tableofcontents

\section{Introduction}

\subsection{Overview}

In \cite{FG2}, Frenkel and Gaitsgory formulated a conjectural analogue of
Beilinson-Bernstein localization for the affine Grassmannian using
critical level Kac-Moody algebras and geometric Satake. In this paper,
we prove their conjecture.

We highlight \cite{FGi0} and \cite{raskinbeiber} as previous works
in this area. The former proves the affine localization conjecture
after passing to Iwahori equivariant objects on both sides. 
The latter, due to the first author, proves the full conjecture 
for rank $1$ groups $G$.

We refer to the introductions of \cite{FG2}, \cite{FGi0}, and \cite{raskinbeiber}
for motivation and background on the subject. We particularly 
refer to \cite{raskinbeiber}, which is close in spirit in many ways
to the present paper.

\subsection{Methods}

Our methods have some overlap with \cite{raskinbeiber}, but are
different in their core. Specifically, as discussed in 
\cite{raskinbeiber} Section 1.14, we use a recent general
technique for categories with $G((t))$-actions and then specify
to Kac-Moody representations: \emph{categorical Moy-Prasad
theory}. This subject was developed by the second author in \cite{dhy},
where it was already used to study critical level Kac-Moody representations.

Roughly speaking, the major outstanding problem has been to show
essential surjectivity and $t$-exactness of the global sections functor
\[
\Gamma^{\Hecke}:D^{\Hecke_{\fz}}_{\crit}(\Gr_G)\rightarrow\widehat{\fg}_{\crit}\mod_{\reg},
\]
 
\noindent see Section \ref{ss:gamma-hecke}. The idea is that 
\cite{dhy} \emph{almost} shows that $\widehat{\fg}_{\crit}\mod_{\reg}$ is of
depth $0$. We would then be reduced to checking essential surjectivity
on Iwahori invariants, where it was treated by Frenkel-Gaitsgory. 

The issue is in the \emph{almost}. Specifically, 
\cite{dhy} showed that the subcategory 
$\widehat{\fg}_{\crit}\mod_{\widehat{\reg}} \subset 
\widehat{\fg}_{\crit}\mod$ generated by $\widehat{\fg}_{\crit}\mod_{\reg}$
under colimits has depth $0$. But $\widehat{\fg}_{\crit}\mod_{\reg}$
itself is finicky: for instance, the forgetful functor 
$\widehat{\fg}_{\crit}\mod_{\reg} \to \widehat{\fg}_{\crit}\mod$ is
not actually conservative. Moreover, it is 
not a priori clear that $\widehat{\fg}_{\crit}\mod_{\reg}$ carries a
$G((t))$-action, so it is not clear that depth is a meaningful notion 
for it. We refer to \cite{raskinbeiber} Section 1.22 for related discussion.

This paper is dedicated to solving those problems. At the same time,
we also address the $t$-exactness, which is intimately related 
to the above subtleties. 

This overall strategy is in contrast to \cite{raskinbeiber}, which relied 
heavily on Whittaker techniques in place of Moy-Prasad theory.

\subsection{Notation and Background}

Throughout this paper, we fix a field $k$ of characteristic $0$. We also fix a split reductive group $G$ over $k$ and a Borel subgroup $B$ of $G$. The unipotent radical of $B$ will be denoted by $N$, and the Cartan will be denoted by $T$. The weight and coweight lattices will be denoted by $X^*(T)$ and $X_*(T)$, respectively. The Lie algebras of $G,B$, etc.\ will be denoted by $\fg,\mathfrak{b}$, etc.

We use \emph{category} to mean \emph{$\infty$-category} and \emph{DG category}
to mean \emph{presentable stable category over $k$}. 
Similarly, all t-structures will be assumed to be right complete and 
compatible with filtered colimits (in particular, accessible).
We let $\DGCat$ denote Lurie's symmetric monoidal category of DG categories; 
with denote the binary operation of its tensor product by $\otimes$.

We will frequently invoke the theory of D-modules on infinite-dimensional spaces developed in \cite{dmodinf}. Occasionally, we will also use the theory of ind-coherent sheaves on such spaces, for which the reader is referred to Section 6 of \cite{semiinf}.

We will be interested in group actions on categories, both weak and strong. 
We understand group actions on categories to always imply actions
on \emph{DG} categories. 
This theory has been developed in \cite{Beraldo} and \cite{semiinf}. In general, when we speak of a group action without further specification, we are referring to a strong action. We will also need the following notation for convolution. If $\sC$ is a category acted on by $G$ (or by some other group), and $K\subseteq G$ is some subgroup, then we have a convolution functor $D(G/K)\otimes \sC^K\rightarrow \sC$. This functor will be denoted by $\overset{K}{\star}$. The overset $K$ will often be omitted when the subgroup is clear from context.

\subsection{Acknowledgements}

We thank Dima Arinkin, Sasha Beilinson, Dario Beraldo, David Ben- Zvi, Roman Bezrukavnikov, Justin Campbell, Kylin Chen, Vladimir Drinfeld, Gurbir Dhillon, Yuchen Fu, Dennis Gaitsgory, and Ivan Mirkovic for useful discussions related to this work.

\section{Moy-Prasad Preliminaries}

\subsection{Categorical depth filtration}

This subsection and the next give a quick overview of the results in Sections 2 and 3 of \cite{dhy}. We recommend looking there for a more detailed exposition.

First, we need to define Moy-Prasad subgroups. Let $\fg=\oplus\fg_{\alpha}$ be the weight decomposition of $\fg$, and fix a point $x\in X_*(T)\otimes\mathbb{R}$ as well as a nonnegative real number $r$. Then $K_{x,r}$ and $K_{x,r+}$ are defined as the exponentiations of the Lie subalgebras
\[
\fk_{x,r}=\displaystyle\bigoplus_{\langle\alpha,x\rangle+i\geq r}\fg_{\alpha}t^i
\]
and
\[
\fk_{x,r+}=\displaystyle\bigoplus_{\langle\alpha,x\rangle+i> r}\fg_{\alpha}t^i.
\]

As shorthand, we write $P_x$ for $K_{x,0}$. This is always a parabolic subgroup of $G((t))$, justifying the notation. We also write $L_x$ for the reductive quotient $K_{x,0}/K_{x,0+}$. 

Let $\sC$ be a $G((t))$-category, i.e., a category equipped with a strong $G((t))$-action. Then Lemma 2.2 of \cite{dhy} (which in turn is based on the main theorem of \cite{BGO}) states:

\begin{lemma}\label{bgolemma}
Let $\sC$ be a category acted on by $G((t))$. Then there is a natural $G((t))$-equivariant fully faithful embedding
\[
a_{x,r}:D(G((t))/K_{x,r+})\underset{D(K_{x,r+}\backslash G((t))/K_{x,r+})}{\otimes} \sC^{K_{x,r+}}\rightarrow \sC
\]
with a continuous $G((t))$-equivariant right adjoint $a_{x,r}^R$.
\end{lemma}

We denote the essential image of $a_{x,r}$ by $\sC_{K_{x,r+}\gen}$. In the case where this is all of $\sC$, we say that $\sC$ is $K_{x,r+}$-generated. The following gives an equivalent characterization of $\sC_{K_x,r+\gen}$.

\begin{lemma}\label{smallestsubcat}
The subcategory $\sC_{K_{x,r+}\gen}\subseteq \sC$ is the smallest full subcategory of $\sC$ which both contains $\sC^{K_{x,r+}}$ and inherits a $G((t))$-action from $\sC$.
\end{lemma}

\begin{proof}
Let $\sD$ be another full $G((t))$-subcategory of $\sC$ which contains $\sC^{K_{x,r+}}$. Then we must necessarily have $\sC^{K_{x,r+}}=\sD^{K_{x,r+}}$, so $\sC_{K_{x,r+}\gen}=\sD_{K_{x,r+}\gen}\subseteq \sD$, as desired.
\end{proof}

\begin{definition}\label{depthdefinition}
Let $\sC$ be a category acted on by $G((t))$. The \emph{depth filtration} on $\sC$ is the $\mathbb{R}_{\geq 0}$-indexed filtration with $\sC_{\leq r}$ the smallest full DG subcategory of $\sC$ containing $\sC_{K_{x,r+}\gen}$ for every $x\in X_*(T)\otimes\mathbb{R}$.
\end{definition}

Let us list some properties of the depth filtration.

\begin{proposition}\label{depthproperties}
\begin{enumerate}
\item For a fixed $G$, there is a discrete set of rational numbers such that, for any $G((t))$-category $\sC$, the depth filtration on $\sC$ can only jump at those numbers.
\item The depth filtration is exhaustive, i.e., we have
\[
\on{colim} \sC_{\leq r}\cong \sC.
\]
\item Assume that $\sC_{\leq r}=\sC$. Then for any other $G((t))$-category $\sD$ with $\sD_{\leq r}=0$, we have $\sC\otimes_{G((t))}\sD\cong 0$.
\end{enumerate}
\end{proposition}

\begin{proof}
\begin{enumerate}
\item This is Lemma 2.6 of \cite{dhy}.
\item This is proven in the discussion after Definition 2.5 of \cite{dhy}.
\item This is Lemma 2.9 of \cite{dhy}.
\end{enumerate}
\end{proof}

\subsection{Moy-Prasad generators}

The main technical tool of \cite{dhy} is a construction of 2-categorical generators for each depth. Let us rephrase this result in a form amenable to our purposes.

Fix some depth $r>0$. For any $G((t))$-category $\sC$, we have an action of $D(K_{x,r}/K_{x,r+})$ (with the convolution monoidal structure) on $\sC^{K_{x,r+}}$. A priori, $K_{x,r}/K_{x,r+}$ has only the structure of a group, but from unwinding the definitions we see that in fact it has a $k$-vector space structure. Thus, the Fourier-Deligne transform of D-modules (see e.g., Section 5.1 of \cite{Beraldo}) identifies the monoidal category $D(K_{x,r}/K_{x,r+})$ with the monoidal category $D((K_{x,r}/K_{x,r+})^*)$.

The vector space $K_{x,r}/K_{x,r+}$ is naturally equipped with an $L_x$ action, and hence so is the dual vector space $(K_{x,r}/K_{x,r+})^*$. Let $(K_{x,r}/K_{x,r+})^{*,\circ}$ denote the locus of GIT-semistable elements, i.e., elements whose $L_x$ orbit does not contain zero in its closure. Equivalently, $(K_{x,r}/K_{x,r+})^{*,\circ}$ is the complement of the preimage of zero along $(K_{x,r}/K_{x,r+})^*\rightarrow (K_{x,r}/K_{x,r+})^*//L_x$. This is an open subvariety of $(K_{x,r}/K_{x,r+})^*$. We define
\[
\sC^{K_{x,r+},\circ}\cong \sC^{K_{x,r+}}\underset{D((K_{x,r}/K_{x,r+})^*)}{\otimes}
D((K_{x,r}/K_{x,r+})^{*,\circ}).
\]
Note that the operation of $(K_{x,r+},\circ)$-invariants commutes with both limits and colimits. 

The complement of $(K_{x,r}/K_{x,r+})^{*,\circ}$, i.e., the closed locus of GIT-unstable elements, will be denoted by $(K_{x,r}/K_{x,r+})^{*,us}$. In analogy with the definition of $\sC^{K_{x,r+},\circ}$, we define
\[
\sC^{K_{x,r+},us}\cong \sC^{K_{x,r+}}\underset{D((K_{x,r}/K_{x,r+})^*)}{\otimes}D((K_{x,r}/K_{x,r+})^{*,us}).
\]
The relevance of these definitions is explained by the following theorem.

\begin{theorem}[\cite{dhy}]\label{unrefined}
Let $\sC$ be a $G((t))$-category. Then the subcategories $\sC_{<r}$ and $\sC_{\leq r}$ coincide if and only if we have, for all $x\in X_*(T)\otimes\mathbb{R}$,
\[
\sC^{K_{x,r+},\circ}\cong 0.
\]
\end{theorem}

\begin{proof}
Without loss of generality, we can assume that $\sC_{\leq r}=\sC$. By Theorem 3.2 of \cite{dhy}, we have $D(G((t)))^{K_{x,r+},\circ}\cong D(G((t))/K_{x,r+})_{\geq r}$. In particular, part (3) of Proposition \ref{depthproperties} implies that
\[
\sC_{<r}^{K_{x,r+},\circ}\cong \sC_{<r}\underset{G((t))}{\otimes}D(G((t))/K_{x,r+})_{\geq r}\cong 0.
\]
We can therefore quotient $\sC$ by $\sC_{<r}$ and reduce to the case where $\sC_{<r}\cong 0$. It remains to show, under our assumptions, that $\sC$ itself is trivial if and only if all the categories $\sC^{K_{x,r+},\circ}$ are trivial.

The only if direction is now obvious, so let us treat the if direction. Assume $\sC$ is not trivial. Then some $\sC^{K_{x,r+}}$ is nontrivial. We claim that $\sC^{K_{x,r+},\circ}\cong \sC^{K_{x,r+}}$, or equivalently, that
\[
\sC^{K_{x,r+},us}\cong \sC\underset{G((t))}{\otimes}D(G((t)))^{K_{x,r+},us}\cong 0.
\]
This again follows from part (3) of Proposition \ref{depthproperties} and Theorem 3.2 of \cite{dhy}.
\end{proof}

Let us give some properties of the unstable locus. They will come in handy later when we need variations of Theorem \ref{unrefined}. We note that Lemma \ref{geometricinput} is a key geometric input for Theorem 3.2 of \cite{dhy} (and hence for Theorem \ref{unrefined} as well.)

\begin{lemma}\label{geometricinput}
There is an equality of subvarieties
\[
(K_{x,r}/K_{x,r+})^{*,us}=L_x\cdot(\displaystyle\bigcup_{y\mid K_{x,r+}\subseteq K_{y,r}\subseteq K_{x,r}}(K_{y,r}/K_{x,r+})^{\perp}).
\]
\end{lemma}

\begin{proof}
Let $p$ be a point of $(K_{x,r}/K_{x,r+})^{*,us}$. By the Hilbert-Mumford criterion, there is a one-parameter subgroup $\mathbb{G}_m\rightarrow L_x$ which contracts $p$ to the origin. This subgroup is necessarily a conjugate of a one-parameter subgroup of $T$ corresponding to some $\beta\in X_*(T)$. Equivalently, there is some point $q\in L_x\cdot p$ which is contracted by $\beta$, so there is an inclusion
\[
q\in \displaystyle\bigoplus_{\langle \beta,\alpha\rangle > 0}\fg_{-\alpha}^*t^{i_{\alpha}}\bigoplus.
\]

Then for sufficiently small $\epsilon > 0$, if we take $y=x+\epsilon\beta$, we have $K_{x,r+}\subseteq K_{y,r}\subseteq K_{x,r}$ and
\[
q\in (K_{y,r}/K_{x,r+})^{\perp}.
\]
This implies the inclusion $p\in L_x\cdot(K_{y,r}/K_{x,r+})^{\perp}$, as desired.
\end{proof}

\begin{lemma}\label{finiteorbits}
The variety $(K_{x,r}/K_{x,r+})^{*,us}$ is the union of a finite number of $L_x$-orbits.
\end{lemma}

\begin{proof}
We rewrite the statement in terms of graded Lie algebras. Express $r$ as a fraction $\frac{p}{q}$ in lowest terms. Then we have a $\mathbb{Z}/q\mathbb{Z}$ grading on $\fg$, defined by
\[
\fg_i=\displaystyle\bigoplus_{\langle x,\alpha\rangle\equiv\frac{pi}{q}\pmod{1}}\fg_{\alpha}.
\]
This grading is compatible with the Lie algebra structure in the sense that $[\fg_i,\fg_j]\subseteq\fg_{i+j}$.

In this language, the vector space $K_{x,r}/K_{x,r+}$ can be identified with $\fg_1$, and $L_x$ can be identified with the exponentiation of $\fg_0$. The lemma now follows from Theorem 4 of \cite{Vinberg}.
\end{proof}

The rest of this section will be devoted to the case of $r=0$. The naive analogue of Theorem \ref{unrefined} would say that $\sC_{\leq 0}$ is trivial if and only if $\sC^{K_{x,0+}}$ is trivial for all $x$. This is easy to prove, but we will need a more refined statement.

\subsection{Categorical representations of reductive groups}

First we need to study the categorical representation theory of (finite type) reductive groups. Let $H$ be a reductive algebraic group over $k$. (We use the letter $H$ here to emphasize that it may differ from the group $G$ fixed in the rest of the paper.) The Borel, Cartan, etc. subgroups of $H$ will be denoted by $B_H$, $T_H$, etc.

Fix a $H$-category $\sC$. Our main lemma relates the categories $\sC^{H,w}\cong \sC\otimes_{H}\fh\mod$ and $(\sC^{N_H})^{T_H,w}\cong \sC\otimes_{H}D(H/N_H)^{T_H,w}$. The Harish-Chandra isomorphism gives an action of $\QCoh(\ft_{\fh}^*//W_H)$ on $\sC^{H,w}$, where the $W_H$ action on $\ft_{\fh}^*$ is the dot action. On the other hand, the tangent vector fields for the right action of $T_H$ on $H/N_H$ give an action of $\QCoh(\ft_{\fh}^*)$ on $(\sC^{N_H})^{T_H,w}$.

Let $S$ be the collection of elements $\check{\alpha}+\langle\rho,\check{\alpha}\rangle+n\in\Sym\ft_{\fh}$ where $\check{\alpha}$ ranges over positive coroots of $H$ and $n$ ranges over positive integers. Then we define an affine scheme $\ft_{\fh}^{*,\circ}$ as $\Spec$ of the localization $\Sym\ft_{\fh}[S^{-1}]$. 

\begin{lemma}\label{finitebeiber}
There are adjoint functors
\[
\on{Loc}:\sC^{H,w}\underset{\QCoh(\ft_{\fh}^*//W_H)}{\otimes}\QCoh(\ft_{\fh}^{*,\circ})\rightleftarrows(\sC^{N_H})^{T_H,w}\underset{\QCoh(\ft_{\fh}^*)}{\otimes}\QCoh(\ft_{\fh}^{*,\circ}):\Gamma
\]
with $\on{Loc}$ fully faithful. Furthermore, if $\sC$ has a t-structure compatible with the $H$-action, then $\Gamma$ is t-exact for the t-structures induced by Lemma \ref{flattstructure}.
\end{lemma}

\begin{proof}
To construct $\Gamma$ and $\on{Loc}$, it suffices to consider the universal case of $\sC\cong D(G)$. In this case, we are looking for an adjoint pair
\[
\on{Loc}:\fh\mod\underset{\QCoh(\ft_{\fh}^*//W_H)}{\otimes}\QCoh(\ft_{\fh}^{*,\circ})\rightleftarrows(D(H/N_H))^{T_H,w}\underset{\QCoh(\ft_{\fh}^*)}{\otimes}\QCoh(\ft_{\fh}^{*,\circ}):\Gamma.
\]
These are provided by the global sections and localization functors of \cite{BB}. By \emph{loc. cit}., in this case $\on{Loc}$ is fully faithful and $\Gamma$ is t-exact.

For general $\sC$, we must still have $\Gamma\circ\on{Loc}\cong\on{id}$ and so $\on{Loc}$ is still fully faithful. By Lemma \ref{rightexacttensor} and Lemma \ref{compatibletensor}, $\on{Loc}$ and $\Gamma$ are right t-exact. As $\Gamma$ has left adjoint $\on{Loc}$, we see that $\Gamma$ must also be left t-exact, as desired.
\end{proof}

\begin{corollary}\label{n-generation}
The functor
\[
D(H/N_H)\underset{D(N_H\backslash H/N_H)}{\otimes}\sC^{N_H}\rightarrow \sC
\]
is an equivalence.
\end{corollary}

\begin{proof}
By \cite{BGO}, this functor is fully faithful. Denote the quotient category by $\sD$. We need to show that $\sD$ is trivial.

A short calculation shows that $\sD^{N_H}\cong 0$. Thus, fully faithfulness of $\on{Loc}$ in Lemma \ref{finitebeiber} implies that $\sD^{H,w}\otimes_{\QCoh(\ft_{\fh}^*//W_H)}\QCoh(\ft_{\fh}^{*,\circ})$ is trivial. We claim that the functor
\[
\sD^{H,w}\rightarrow \sD^{H,w}\underset{\QCoh(\ft_{\fh}^*//W_H)}{\otimes} \QCoh(\ft_{\fh}^{*,\circ})
\]
is conservative. As $\QCoh(\ft_{\fh}^{*,\circ})$ is self dual as a $\QCoh(\ft_{\fh}^*//W_H)$-module, this functor can be rewritten as
\[
\on{Hom}_{\QCoh(\ft_{\fh}^*//W_H)}(\QCoh(\ft_{\fh}^*//W_H),\sD^{H,w})\rightarrow\on{Hom}_{\QCoh(\ft_{\fh}^*//W_H)}(\QCoh(\ft_{\fh}^{*,\circ}),\sD^{H,w}),
\]
which is conservative if the pushforward functor $f_*:\QCoh(\ft_{\fh}^{*,\circ})\rightarrow\QCoh(\ft_{\fh}^*//W_H)$ generates under colimits. As $f:\ft_{\fh}^{*,\circ}\rightarrow\ft_{\fh}^*//W_H$ is faithfully flat, the image of $f_*$ contains a skyscraper sheaf at each geometric point. And because $\ft_{\fh}^*//W_H$ is smooth, $\QCoh(\ft_{\fh}^*//W_H)$ is generated under colimits by skyscraper sheaves, giving the conservativeness.

In our case, this shows that $\sD^{H,w}$ is trivial. By 1-affineness of $BH$ (Theorem 2.2.2 of \cite{1affine}), this implies that $\sD$ is trivial, as desired.
\end{proof}

The t-structures appearing in Lemma \ref{finitebeiber} have the following source.

\begin{lemma}\label{flattstructure}
Let $f:\Spec B\rightarrow\Spec A$ be a flat morphism of affine schemes. Let $\sC$ be a category equipped with a t-structure and an action of $A{\mod}$. Then there is a unique t-structure on ${\sC\otimes_{A\mod}B\mod}$ for which the pullback and pushforward functors
\[
f_{\sC}^*:\sC\rightleftarrows \sC\underset{A\mod}{\otimes} B\mod:f_{\sC,*}
\]
are both t-exact. Also, $f_{\sC,*}$ is conservative.
\end{lemma}

\begin{proof}
To show conservativeness, note that this property of a right adjoint is preserved under tensor product. Thus, it suffices to check that $B\mod\rightarrow A\mod$ is conservative, which is obvious. Furthermore, by Barr-Beck-Lurie, the functor $f_{\sC,*}$ is monadic.

The existence of the desired t-structure on will thus follow from t-exactness of the monad $f_{\sC,*}\circ f_{\sC}^*$. This monad can be identified with the tensor product functor $-\otimes_A B: \sC\rightarrow \sC$, where $B$ is treated as an $A$-module. By Lazard's theorem, $B$ is a filtered colimit of finite free $A$-modules. Tensoring by a finite free $A$-module is clearly t-exact, so $-\otimes_A B$ is also t-exact, as desired.
\end{proof}

\subsection{Depth zero}

Now we can analyze the $r=0$ case more fully. Recall that we denote the Iwahori subgroup by $I$ and its unipotent radical by $I_0$.

\begin{theorem}\label{unrefineddepthzero}
Let $\sC$ be a $G((t))$-category. We have an equality of subcategories
\[
\sC_{\leq 0} = \sC_{I_0\gen} \subset \sC.
\]
In particular, $\sC_{\leq 0}$ is trivial if and only if $\sC^{I_0}$ is trivial.
\end{theorem}

\begin{proof}
As $I_0$ is a subgroup of the form $K_{x,0+}$, the inclusion $\sC_{I_0\gen}\subseteq \sC_{\leq 0}$ is clear. By Lemma \ref{smallestsubcat}, it suffices to show that, for all $x\in X_*(T)$, $\sC^{K_{x,0+}}\subseteq \sC_{I_0\gen}$.

Note that for sufficiently small $\epsilon > 0$, the subgroup $K_{x+\epsilon\check{\rho},0+}$ does not depend on $\epsilon$. (Recall that $\check{\rho}$ denotes the half-sum of positive coroots.) Indeed, unwinding the definitions, we have an equality of Lie algebras
\begin{align*}
\fk_{x+\epsilon\check{\rho},0+}&=\displaystyle\bigoplus_{\langle\alpha,x+\epsilon\check{\rho},\rangle+i >0}\fg_{\alpha}t^i\\
&=\displaystyle\bigoplus_{\langle\alpha,x\rangle+i>0}\fg_{\alpha}t^i\displaystyle+\bigoplus_{\substack{\langle\alpha,x\rangle+i\geq 0 \\ \langle\alpha,\check{\rho}\rangle>0}}\fg_{\alpha}t^i\\
&=\fk_{x,0+}+(\fk_{x,0}\cap\fn((t)))
\end{align*}
giving the corresponding identity of subgroups
\[
K_{x+\epsilon\check{\rho},0+}=K_{x,0+}(P_x\cap N((t))).
\]

First, we show that $\sC^{K_{x+\epsilon\check{\rho},0+}}$ is a subcategory of $\sC_{I_0\gen}$.  Indeed, $K_{x+\epsilon\check{\rho},0}$ will be a conjugate $gIg^{-1}$ of the Iwahori subgroup, and $K_{x+\epsilon\check{\rho},0+}$ will be the corresponding conjugate $gI_0g^{-1}$ of $I_0$. We then have
\begin{align*}
\sC^{K_{x+\epsilon\check{\rho},0+}}&=g\sC^{I_0}\\
&\subseteq \sC_{I_0\gen}
\end{align*}
as desired.

To go from $K_{x+\epsilon\check{\rho},0+}$-invariants to $K_{x,0+}$-invariants, recall that $L_x\cong K_{x,0}/K_{x,0+}$ acts on $\sC^{K_{x,0+}}$. Let $N_x$ denote the subgroup $(K_{x,0+}(P_x\cap N((t))))/K_{x,0+}$ of $L_x$. As $L_x$ is a reductive group with maximal unipotent subgroup $N_x$, Corollary \ref{n-generation} is applicable. Thus, we have
\begin{align*}
\sC^{K_{x,0+}}&\cong (\sC^{K_{x,0+}})^{N_x}\underset{D(N_x\backslash L_x/N_x)}{\otimes} D(L_x/N_x)\\
&\cong \sC^{K_{x+\epsilon\check{\rho},0+}}\underset{D(N_x\backslash L_x/N_x)}{\otimes} D(P_x/K_{x+\epsilon\check{\rho},0+})\\
&\subseteq \sC_{I_0\gen}
\end{align*}
giving the theorem.
\end{proof}

\begin{corollary}\label{vanishingcriterion}
Let $\sC$ be a $G((t))$-category. Assume that, for all $x\in X_*(T)\otimes\mathbb{R}$ and $r>0$,
\[
\sC^{K_{x,r+},\circ}\cong 0
\]
and also that
\[
\sC^{I_0}\cong 0.
\]
Then $\sC$ must be trivial.
\end{corollary}

\begin{proof}
By Theorem \ref{unrefineddepthzero}, we have $\sC_{\leq 0}\cong 0$. For $r>0$, Theorem \ref{unrefined} tells us that $\sC_{<r}\cong \sC_{\leq r}$. As the depth filtration only jumps at a discrete set (part 1 of Theorem \ref{depthproperties}), this implies that $\sC_{\leq r}\cong 0$ for all $r$. And as the depth filtration is exhaustive (part 2 of Theorem \ref{depthproperties}), we see that $\sC$ itself must be trivial, as desired.
\end{proof}

\section{Kac-Moody categories}

To state the main theorem of this paper, we will need to define various categories of representations of the affine Lie algebra $\widehat{\fg}$. Recall that $\widehat{\fg}$ is defined as a central extension of the loop algebra $\fg((t))$, and that central characters are in bijection with invariant bilinear forms on $\fg$. For the rest of this paper, we will only be concerned with critical level representations, i.e., those $\widehat{\fg}$-modules where the center acts via the character corresponding to negative one half times the Killing form. Notationally, this will be reflected by a subscript, e.g., $\widehat{\fg}_{\crit}$.

This section is effectively a quick overview of the results of Sections 6-7 of \cite{raskinbeiber}, and we will often refer there for proofs.

\subsection{The category $\widehat{\fg}_{\crit}\mod$}

Our starting point is the category $\widehat{\fg}_{\crit}\mod$ constructed in, e.g., \cite{semiinf}. As justification for its name, its heart $\widehat{\fg}_{\crit}\mod^{\heartsuit}$ is the abelian category of discrete critical $\widehat{\fg}$-representations. However, we warn the reader that $\widehat{\fg}_{\crit}\mod$ does not coincide with the derived category of its heart; rather, it is a ``renormalization" of this derived category.

More precisely, consider the subcategory $\widehat{\fg}_{\crit}\mod^c\subseteq D(\widehat{\fg}_{\crit}\mod^{\heartsuit})$ generated under finite limits by modules of the form $\on{ind}_{\fk}^{\widehat{\fg}}\mathbb{C}$, where $\fk$ is the Lie subalgebra corresponding to a pro-unipotent compact open subgroup $K\subset G((t))$. Then $\widehat{\fg}_{\crit}\mod$ is defined\footnote{Actually, in \cite{semiinf}, a different definition is given which is easier to work with. The two definitions are equivalent by Lemma 9.13.1 of \emph{loc. cit}.} to be the ind-completion of $\widehat{\fg}_{\crit}\mod^c$. Kan extension gives a functor
\[
\widehat{\fg}_{\crit}\mod\rightarrow D(\widehat{\fg}_{\crit}\mod^{\heartsuit})
\]
which induces an equivalence on bounded below subcategories. However, $\widehat{\fg}_{\crit}\mod$ contains objects in cohomological degree $-\infty$ which are sent to the zero object in $D(\widehat{\fg}_{\crit}\mod^{\heartsuit})$.

The advantage of this renormalization is that $\widehat{\fg}_{\crit}\mod$ admits a strong action of $G((t))$. There is also an additional structure which we will need. Let $\fZ$ denote the center of the universal enveloping algebra $U(\widehat{\fg}_{\crit})$. This is canonically a pro-vector space, and the Feigin-Frenkel theorem identifies it with functions on an ind-affine ind-scheme $\Op_{\check{G}}$, which parametrizes $\check{G}$-opers on the punctured formal disc. (For more on opers and Feigin-Frenkel, the reader is referred to Sections 1 and 5 of \cite{FG2}, respectively.)

There is a canonical action of $\fZ$ on any $\widehat{\fg}_{\crit}$-module $M$. Equivalently, there is a canonical ind-coherent sheaf on $\Op_{\check{G}}$ whose space of global sections is isomorphic to $M$. The correct categorical enhancement of this statement was identified in \cite{semiinf}:

\begin{theorem}[\cite{semiinf} Theorem 11.18.1]
There is a canonical coaction of $\on{IndCoh}^*(\Op_{\check{G}})$ on $\widehat{\fg}_{\crit}\mod$. Furthermore, this coaction is compatible with the action of $G((t))$.
\end{theorem}

We remark that the coaction map
\[
\widehat{\fg}_{\crit}\mod\rightarrow \widehat{\fg}_{\crit}\mod\otimes\on{IndCoh}^*(\Op_{\check{G}})
\]
corresponds to restriction of modules along the algebra homomorphism $\fZ\times U(\widehat{\fg}_{\crit})\rightarrow U(\widehat{\fg}_{\crit})$.

\subsection{The category $\widehat{\fg}_{\crit}\mod_{\reg}$}

There is a natural closed subscheme $\Op_{\check{G}}^{\reg}$ of $\Op_{\check{G}}$, corresponding to regular opers, i.e., those that extend to the formal disc. We are interested in studying categories of $\widehat{\fg}$-representations for which the action of $\fZ$ factors through $\sO(\Op_{\check{G}}^{\reg})$. The most famous example of such a representation is the vacuum module $\mathbb{V}_{\crit}$, defined as the critical level induction from $\fg[[t]]$ of the trivial representation.

The abelian category of such representations is easy to define. However, as in the previous subsection, there are subtleties of renormalization in defining the derived category. We will define two categories, denoted by $\widehat{\fg}_{\crit}\mod_{\reg}$ and $\widehat{\fg}_{\crit}\mod_{\regnaive}$. Both of these have t-structures with the desired heart, but they will differ at $-\infty$. As suggested by the notation, $\widehat{\fg}_{\crit}\mod_{\reg}$ is the better behaved category and the real target of our investigation, while $\widehat{\fg}_{\crit}\mod_{\regnaive}$ appears as a useful technical intermediate.

First we will define $\widehat{\fg}_{\crit}\mod_{\regnaive}$. Let us recall the notion of a cotensor product. Let $\sA$ be a comonoidal category, and take a right $\sA$-comodule $\sB$ and a left $\sA$-comodule $\sC$. Then the cotensor product $\sB\overset{\sA}{\otimes}\sC$ is defined as the totalization of the cosimplicial category
\[
\sB\otimes \sC\rightrightarrows \sB\otimes \sA\otimes \sC \rightthreearrow \sB\otimes \sA\otimes \sA\otimes \sC\cdots.
\]

Then we define
\[
\widehat{\fg}_{\crit}\mod_{\regnaive}\cong\widehat{\fg}_{\crit}\mod\overset{\IndCoh^*(\Op_{\check{G}})}{\otimes}\IndCoh^*(\Op_{\check{G}}^{\reg}).
\]

Denote the closed embedding $\Op_{\check{G}}^{\reg}\rightarrow\Op_{\check{G}}$ by $i$. By Lemmas 6.17.1-2 of \cite{semiinf}, we have a pair of adjoint functors
\[
\IndCoh^*(\Op_{\check{G}}^{\reg})\rightleftarrows\IndCoh^*(\Op_{\check{G}}).
\]
Tensoring, we get another adjoint pair
\[
\widehat{\fg}_{\crit}\mod_{\regnaive}\rightleftarrows\widehat{\fg}_{\crit}\mod.
\]
We will abuse notation and denote both adjoint pairs by $(i_*,i^!)$.

\begin{theorem}
The functor $i_*:\widehat{\fg}_{\crit}\mod_{\regnaive}\rightarrow\widehat{\fg}_{\crit}\mod$ is comonadic. Thus, $\widehat{\fg}_{\crit}\mod_{\regnaive}$ admits a unique t-structure making $i_*$ t-exact.
\end{theorem}

\begin{proof}
This is part of Proposition 6.6.1. of \cite{raskinbeiber}.
\end{proof}

Let $\widehat{\fg}_{\crit}\mod_{\reg}^c\subseteq\widehat{\fg}_{\crit}\mod_{\regnaive}$ denote the full subcategory of objects $M$ such that $i_*M$ is compact. We define $\widehat{\fg}_{\crit}\mod_{\reg}$ as the ind-completion of $\widehat{\fg}_{\crit}\mod_{\reg}^c$. It comes with a natural t-structure and a functor $p:\widehat{\fg}_{\crit}\mod_{\reg}\rightarrow\widehat{\fg}_{\crit}\mod_{\regnaive}$.

\begin{theorem}\label{eventuallyequivalence}
The functor $p$ is t-exact and induces an equivalence on eventually coconnective subcategories.
\end{theorem}

\begin{proof}
This is Lemma 6.9.3 of \cite{raskinbeiber}.
\end{proof}

\begin{remark}
One technical problem with the category $\widehat{\fg}_{\crit}\mod_{\reg}$ is that it is a priori unclear that it admits a $G((t))$-action. Once our Theorem \ref{main} is known, there is a clear such action. But for the proof of Theorem \ref{main}, we will need to juggle this category with its naive version $\widehat{\fg}_{\crit}\mod_{\regnaive}$, which has a $G((t))$ action by virtue of its construction.
\end{remark}

\subsection{The functor $\Gamma^{\Hecke}$}\label{ss:gamma-hecke}

Now we would like to relate these categories of Kac-Moody representations to geometric categories. Let $\Gr_G$ denote the affine Grassmannian $G((t))/G[[t]]$. To start, note that
\[
\on{Hom}_{G((t))}(D_{\crit}(\Gr_G),\widehat{\fg}_{\crit}\mod_{\regnaive})\cong \widehat{\fg}_{\crit}\mod_{\regnaive}^{G[[t]]},
\]
so the vacuum module $\mathbb{V}_{\crit}$ defines a map $D_{\crit}(\Gr_G)\rightarrow\widehat{\fg}_{\crit}\mod_{\regnaive}$. Combining this with the $\QCoh(\Op_{\check{G}}^{\reg})$ action constructed in the previous subsection, we get a map
\begin{equation}\label{tensormap}
D_{\crit}(\Gr_G)\otimes\QCoh(\Op_{\check{G}}^{\reg})\rightarrow\widehat{\fg}_{\crit}\mod_{\regnaive}.
\end{equation}

Recall that geometric Satake (see \cite{MV}) gives a morphism of monoidal categories
\[
\on{Rep}(\check{G})\rightarrow D_{\crit}(\Gr_G)^{G[[t]]}\cong\on{End}_{G((t))}(D_{\crit}(\Gr_G)),
\]
so we get a right action of $\on{Rep}(\check{G})$ on $D_{\crit}(\Gr_G)$.

There is also a natural action of $\on{Rep}(\check{G})$ on $\QCoh(\Op_{\check{G}}^{\reg})$, coming from a natural map $\Op_{\check{G}}^{\reg}\rightarrow\mathbb{B}\check{G}$. Such a map is equivalent to the data of a $\check{G}$-torsor on $\Op_{\check{G}}^{\reg}$. By definition, there is a canonical family of opers (and hence a canonical $\check{G}$-torsor) on $\Spec(\sO(\Op_{\check{G}}^{\reg})[[t]])$; restricting to $\Op_{\check{G}}^{\reg}$ gives the desired torsor.

We now claim that the map (\ref{tensormap}) naturally factors as
\[
D_{\crit}(\Gr_G)\otimes\QCoh(\Op_{\check{G}}^{\reg})\rightarrow D_{\crit}(\Gr_G)\underset{\on{Rep}(\check{G})}{\otimes}\QCoh(\Op_{\check{G}}^{\reg})\rightarrow\widehat{\fg}_{\crit}\mod_{\regnaive},
\]
or equivalently, that the map
\[
\QCoh(\Op_{\check{G}}^{\reg})\rightarrow\on{Hom}_{G((t))}(D_{\crit}(\Gr_G),\widehat{\fg}_{\crit}\mod_{\regnaive})\cong\widehat{\fg}_{\crit}\mod_{\regnaive}^{G[[t]]}
\]
can naturally be lifted to a $\on{Rep}(\check{G})$-equivariant map.

This is proven rigorously in Sections 7.9-7.11 of \cite{raskinbeiber}. The key tool is the following theorem of Beilinson-Drinfeld, which immediately implies the above claim at the level of objects.

\begin{theorem}[\cite{quantization}]
Let $V$ be a finite dimensional object in $\on{Rep}(\check{G})^{\heartsuit}$. Denote by $\sS_V$ and $\sP_V$ the corresponding objects of $D(\Gr_G)^{G[[t]]}$ and $\QCoh(\Op_{\check{G}}^{\reg})$, respectively. Then
\[
\sS_V\star\mathbb{V}_{\crit}\cong\sP_V\otimes\mathbb{V}_{\crit},
\]
where $\star$ denotes the convolution action of $D(\Gr_G)^{G[[t]]}$ on $\widehat{\fg}_{\crit}\mod_{\regnaive}$.
\end{theorem}

As shorthand, we will denote $D_{\crit}(\Gr_G)\otimes_{\on{Rep}(\check{G})}\QCoh(\Op_{\check{G}}^{\reg})$ by $D^{\Hecke_{\fz}}_{\crit}(\Gr_G)$. The map $D^{\Hecke_{\fz}}_{\crit}(\Gr_G)\rightarrow\widehat{\fg}_{\crit}\mod_{\regnaive}$ will be denoted by $\Gamma^{\Heckenaive}$.

We now want to lift $\Gamma^{\Heckenaive}$ to a functor $\Gamma^{\Hecke}:D^{\Hecke_{\fz}}_{\crit}(\Gr_G)\rightarrow\widehat{\fg}_{\crit}\mod_{\reg}$. First, we will show that $D^{\Hecke_{\fz}}_{\crit}(\Gr_G)$ is compactly generated. As $\Op_{\check{G}}^{\reg}\rightarrow\mathbb{B}\check{G}$ is affine, we have a pair of continuous $\on{Rep}(\check{G})$-equivariant adjoint functors $\on{Rep}(\check{G})\rightleftarrows\QCoh(\Op_{\check{G}}^{\reg})$. These induce a pair of adjoint functors
\[
\on{ind}^{\Hecke_{\fz}}: D_{\crit}(\Gr_G)\rightleftarrows D^{\Hecke_{\fz}}_{\crit}(\Gr_G):\on{Oblv}^{\Hecke_{\fz}}.
\]

In particular, $\on{ind}^{\Hecke_{\fz}}$ sends compact objects to compact objects. Furthermore, we have:

\begin{lemma}
The objects $\on{ind}^{\Hecke_{\fz}}(\sF)$, for $\sF\in D_{\crit}(\Gr_G)$ compact, generate $D^{\Hecke_{\fz}}_{\crit}(\Gr_G)$.
\end{lemma}

\begin{proof}
As $D_{\crit}(\Gr_G)$ is compactly generated, it suffices to show that the essential image of $\on{ind}^{\Hecke_{\fz}}$ generates $D^{\Hecke_{\fz}}_{\crit}(\Gr_G)$ under colimits. It follows from Lemma \ref{quotientcriterion} that this property is stable under tensor product, so it suffices to show that the pullback map $\on{Rep}(\check{G})\rightarrow\QCoh(\Op_{\check{G}}^{\reg})$ generates under colimits. This follows as the map $\Op_{\check{G}}^{\reg}\rightarrow\mathbb{B}\check{G}$ is affine.
\end{proof}

\begin{corollary}
The functor $\Gamma^{\Heckenaive}$ sends compact objects of $D^{\Hecke_{\fz}}_{\crit}(\Gr_G)$ to objects in $\widehat{\fg}_{\crit}\mod_{\reg}^c$.
\end{corollary}

\begin{proof}
By the preceding lemma, it suffices to show that, for any compact $\sF\in D_{\crit}(\Gr_G)$, the object $\Gamma^{\Heckenaive}(\on{ind}^{\Hecke_{\fz}}(\sF))$ lies in $\widehat{\fg}_{\crit}\mod_{\reg}^c$. Unwinding the definitions, we find that
\[
\Gamma^{\Heckenaive}(\on{ind}^{\Hecke_{\fz}}(\sF))\cong\sF\star\mathbb{V}_{\crit}\in\widehat{\fg}_{\crit}\mod_{\reg}.
\]

We would like to show that $i_*(\sF\star\mathbb{V}_{\crit})\cong\sF\star i_*\mathbb{V}_{\crit}$ is a compact object in $\widehat{\fg}_{\crit}\mod$. As $i_*\mathbb{V}_{\crit}$ is compact, it suffices to show that $\sF\star-:\widehat{\fg}_{\crit}\mod\rightarrow \widehat{\fg}_{\crit}\mod$ preserves compactness, or equivalently, that its right adjoint is continuous.

This right adjoint is explicitly identified in Proposition 22.10.1 of \cite{FG2}. Let $\mathbb{D}\sF$ denote the Verdier dual of $\sF$, and let $\on{inv}\mathbb{D}\sF\in D(G[[t]]\backslash G((t)))$ denote the pullback of $\mathbb{D}\sF$ along the inversion map. Then, for any $G((t))$-category $\sC$, the functors $\sF\star-:\sC\rightarrow \sC^{G[[t]]}$ and $\on{inv}\mathbb{D}\sF\star-:\sC^{G[[t]]}\rightarrow \sC$ form an adjoint pair. As convolution preserves colimits, we arrive at the desired result.
\end{proof}

We thus get a natural functor $D^{\Hecke_{\fz}}_{\crit}(\Gr_G)^c\rightarrow\widehat{\fg}_{\crit}\mod_{\reg}^c$. Define $\Gamma^{\Hecke}$ to be its ind-extension. By definition, we have a commutative triangle

\begin{equation*}
\begin{tikzcd}
D^{\Hecke_{\fz}}_{\crit}(\Gr_G) \arrow[r, "\Gamma^{\Hecke}"] \arrow{rd}[swap]{\Gamma^{\Heckenaive}} & \widehat{\fg}_{\crit}\mod_{\reg} \arrow[d]\\
& \widehat{\fg}_{\crit}\mod_{\regnaive}.
\end{tikzcd}
\end{equation*}

Let us recall some facts about $\Gamma^{\Hecke}$. 

\begin{theorem}[\cite{FG2}, \cite{raskinbeiber}]
The functor $\Gamma^{\Hecke}$ is fully faithful.
\end{theorem}

\begin{proof}\label{fullyfaithful}
This is essentially Theorem 8.7.1 of \cite{FG2}, though they do not work with $\infty$-categories. With our categorical setup, this is stated as Theorem 7.16.1 of \cite{raskinbeiber}, which also gives a new proof.
\end{proof}

For the next result, we need a t-structure on $D^{\Hecke_{\fz}}_{\crit}(\Gr_G)$. Set $D^{\Hecke_{\fz}}_{\crit}(\Gr_G)^{\leq 0}$ to be the full subcategory generated under colimits by the $\on{ind}^{\Hecke_{\fz}}(\sF)$, for $\sF\in D_{\crit}(\Gr_G)^{\leq 0}$. By Proposition 1.4.4.11 of \cite{HA}, this uniquely defines a t-structure.

\begin{theorem}\label{righttexact}
With respect to the previously defined t-structures on $D^{\Hecke_{\fz}}_{\crit}(\Gr_G)$ and $\widehat{\fg}_{\crit}\mod_{\reg}$, $\Gamma^{\Hecke}$ is right t-exact.
\end{theorem}

\begin{proof}
This is Corollary 7.15.3 of \cite{raskinbeiber}.
\end{proof}

\subsection{Localization theorem}

The following is the main theorem of this paper. It confirms one of the main conjectures (more precisely, Main Conjecture 8.5.2) of \cite{FG2}.

\begin{theorem}\label{main}
The functor $\Gamma^{\Hecke}$ is a t-exact equivalence.
\end{theorem}

For $G$ of rank $1$, this was proven in \cite{raskinbeiber}. More generally, for any $G$, \emph{loc. cit}. reduces Theorem \ref{main} to three auxiliary statements, which appear below as Lemmas \ref{generation}, \ref{exactness}, and \ref{boundedness}. All three are proven in \cite{raskinbeiber} only when $G$ has rank $1$; our contribution is to prove them for general $G$. 

\begin{lemma}\label{generation}
Let $\widehat{\fg}_{\crit}\widetilde{\mod}_{\on{reg,naive}}\subseteq\widehat{\fg}_{\crit}\mod_{\on{reg,naive}}$ be the full subcategory generated by $\widehat{\fg}_{\crit}\mod_{\on{reg,naive}}^+$ under colimits. Then the essential image of $\Gamma^{\on{Hecke,naive}}$ lies in $\widehat{\fg}_{\crit}\widetilde{\mod}_{\on{reg,naive}}$ and generates it under colimits.
\end{lemma}

\begin{lemma}\label{exactness}
The functor $\Gamma^{\on{Hecke,naive}}$ is t-exact.
\end{lemma}

\begin{lemma}\label{boundedness}
For every compact open subgroup $K\subseteq G((t))$, the composition
\[
D^{\Hecke_{\fz}}_{\crit}(\Gr_G)^K\rightarrow D^{\Hecke_{\fz}}_{\crit}(\Gr_G)\rightarrow\widehat{\fg}_{\crit}\mod_{\reg}
\]
is left t-exact up to shift.
\end{lemma}

We refer to \cite{raskinbeiber} Section 7.17 for the deduction 
of Theorem \ref{main} from these lemmas.

%\begin{proof}[Proof of Theorem \ref{main}]
%We sketch, following \cite{raskinbeiber}, how the combination of the above three lemmas implies Theorem \ref{main}. First we show that $\Gamma^{\Hecke}$ is t-exact. Theorem \ref{righttexact} guarantees the right t-exactness, so it suffices to treat left t-exactness.
%
%Any object $\sF\in D^{\Hecke_{\fz}}_{\crit}(\Gr_G)$ can be written as $\on{colim}\on{Oblv}\on{Av}_*^K(\sF)$. It follows that it suffices to show left t-exactness of $D^{\Hecke_{\fz}}_{\crit}(\Gr_G)^K\rightarrow\widehat{\fg}_{\crit}\mod_{\reg}$. By Lemma \ref{boundedness}, we have a containment
%\[
%\Gamma^{\Hecke}((D^{\Hecke_{\fz}}_{\crit}(\Gr_G)^K)_{\geq 0})\subseteq (\widehat{\fg}_{\crit}\mod_{\reg})^+.
%\]
%Theorem \ref{eventuallyequivalence} now shows that it suffices to prove that
%\[
%\Gamma^{\Heckenaive}((D^{\Hecke_{\fz}}_{\crit}(\Gr_G)^K)_{\geq 0})\subseteq (\widehat{\fg}_{\crit}\mod_{\regnaive})_{\geq 0},
%\]
%which is true by Lemma \ref{exactness}.
%
%Now we deduce Theorem \ref{main}. As Theorem \ref{fullyfaithful} tells us that $\Gamma^{\Hecke}$ is fully faithful, it suffices to show that $\Gamma^{\Hecke}$ is essentially surjective. Lemma \ref{generation}, combined with t-exactness of $\Gamma^{\Hecke}$ and Theorem \ref{eventuallyequivalence}, implies that $(\widehat{\fg}_{\crit}\mod_{\reg})_{\geq 0}$ lies in the essential image of $\Gamma^{\Hecke}$. As $\widehat{\fg}_{\crit}\mod_{\reg}$ is generated by objects in $\widehat{\fg}_{\crit}\mod_{\reg}^+$, this implies that $\Gamma^{\Hecke}$ is essentially surjective, as desired.
%\end{proof}

Let us explain some philosophy behind the proofs of Lemmas \ref{generation}, \ref{exactness}, and \ref{boundedness}. Each of these Lemmas involves proving some property P of a functor $f:\sC\rightarrow \sD$. In Lemmas \ref{generation} and \ref{exactness}, both $\sC$ and $\sD$ have $G((t))$-actions, while in Lemma \ref{boundedness}, only $\sC$ has a $G((t))$-action. Assume for simplicity that $\sC$ and $\sD$ are both $G((t))$-categories.

Motivated by Corollary \ref{vanishingcriterion}, we will show that $f$ satisfies $P$ if we know that, for all $x\in X_*(T)\otimes\mathbb{R}$ and $r>0$,
\[
f^{K_{x,r+},\circ}:\sC^{K_{x,r+},\circ}\rightarrow D^{K_{x,r+},\circ}
\]
satisfies $P$, as well as that
\[
f^{I_0}:\sC^{I_0}\rightarrow \sD^{I_0}
\]
satisfies $P$. Luckily, each of the involved invariant categories is well-understood, and it will be (relatively) straightforward to show that each of these simpler functors satisfies $P$, proving the desired Lemma.

\begin{remark}
In \cite{raskinbeiber}, a similar line of attack was used, with the invariant categories $\sC^{K_{x,r+},\circ}$ replaced with the category $\Whit(\sC)$ of Whittaker invariants. For $G=PGL_2$, any $G((t))$-category is generated (under the $G((t))$-action) by its $I_0$-invariants and its Whittaker invariants. One impetus for the current paper was the observation (our Corollary \ref{vanishingcriterion}) that this generation statement holds for any group $G$ once one replaces $\Whit(\sC)$ with the categories $\sC^{K_{x,r+},\circ}$.
\end{remark}

Lemma \ref{generation} is easiest, and can be shown with a direct application of Corollary \ref{vanishingcriterion}. The proofs of the other two Lemmas will require reworking each step of the proof of \ref{vanishingcriterion} to apply in a new setting. 

\section{Invariant subcategories}

The goal of this section is to collect some results on invariant subcategories of $D^{\Hecke_{\fz}}_{\crit}(\Gr_G)$ and $\widehat{\fg}_{\crit}\mod_{\on{reg,naive}}$. These results form the concrete input needed for the proof of Theorem \ref{main}.

\subsection{$I_0$-invariants}

The following theorem is the main fact we shall need about the behavior of $\Gamma^{\on{Hecke,naive}}$ on $I_0$-equivariant objects. It is essentially the main theorem of \cite{FGi0}. However, \cite{FGi0} does not state their result in terms of derived categories, so instead we use the phrasing of \cite{raskinbeiber}.

\begin{theorem}[\cite{raskinbeiber} Theorem 8.2.1, \cite{FGi0} Theorem 1.7]\label{i0invariants}
The functor $\Gamma^{\on{Hecke,naive}}$ induces a t-exact equivalence
\[
D^{\Hecke_{\fz}}_{\crit}(\Gr_G)^{I_0,+}\rightarrow\widehat{\fg}_{\crit}\mod_{\on{reg,naive}}^{I_0,+}
\]
on eventually coconnective $I_0$-equivariant categories.
\end{theorem}

For the proof of Lemma \ref{boundedness}, we will need a related statement adapted to $\widehat{\fg}_{\crit}\mod_{\on{reg}}$. Because of the lack of an a priori $G((t))$-action, we cannot speak directly of $I_0$-invariant objects in $\widehat{\fg}_{\crit}\mod_{\on{reg}}$. Even if we could, there would be no convolution functor. Instead, as a proxy, we will look at the images of convolutions with $I_0$-invariant objects in $D^{\Hecke_{\fz}}_{\crit}(\Gr_G)$.

\begin{lemma}\label{i0boundedness}
For any compact object $\sG\in D(G((t))/I_0)$, the functor
\[
\Gamma^{\Hecke}(\sG\star-):D^{\Hecke_{\fz}}_{\crit}(\Gr_G)^{I_0}\rightarrow\widehat{\fg}_{\crit}\mod_{\on{reg}}
\]
is left t-exact up to shift.
\end{lemma}

\begin{remark}
The case of trivial $\sG$ is Lemma 8.2.2 of \cite{raskinbeiber}. The proof we give below is an adaptation of the one in \emph{loc. cit}. to our case.
\end{remark}

\begin{proof}

Choose an integer $r$ for which the conclusion of Lemma \ref{boundedconvolution} is satisfied. We will show that $\Gamma^{\Hecke}(\sG\star-)[r]$ is left t-exact. Because $D^{\Hecke_{\fz}}_{\crit}$ is right complete, it suffices to show that, for $\sF\in D^{\Hecke_{\fz}}_{\crit}(\Gr_G)^{I_0,\heartsuit}$, we have
\[
\Gamma^{\Hecke}(\sG\star\sF)\in(\widehat{\fg}_{\crit}\mod_{\on{reg}})^{\geq r}.
\]

Note that if we replace $\Gamma^{\Hecke}$ with $\Gamma^{\Heckenaive}$, we have
\[
\Gamma^{\Heckenaive}(\sG\star\sF)\cong\sG\star\Gamma^{\Heckenaive}(\sF)\in(\widehat{\fg}_{\crit}\mod_{\regnaive})^{\geq r}
\]
by Theorem \ref{i0invariants} and Lemma \ref{boundedconvolution}. Thus, by Theorem \ref{eventuallyequivalence}, it suffices to show that $\Gamma^{\Hecke}(\sG\star\sF)$ is eventually coconnective.

By Lemma 3.6 and Proposition 3.18 of \cite{FGi0}, $\sF$ can be written as a filtered colimit of objects $\sF_i\in D^{\Hecke_{\fz}}_{\crit}(\Gr_G)^{I_0,\heartsuit}$ admitting finite filtrations with subquotients of the form $\on{ind}^{\Hecke_{\fz}}(\sF_{i,j})\otimes_{\Op_{\check{G}}^{\reg}}\sH _{i,j}$ for $\sF_{i,j}\in D_{\crit}(\Gr_G)^{I_0,\heartsuit}$ and $\sH _{i,j}\in\QCoh(\Op_{\check{G}}^{\reg})^{\heartsuit}$. It thus suffices to show that, for any $i$ and $j$, we have
\[
\Gamma^{\Hecke}(\sG\star\on{ind}^{\Hecke_{\fz}}(\sF_{i,j})\underset{\Op_{\check{G}}^{\reg}}{\otimes}\sH _{i,j})\in(\widehat{\fg}_{\crit}\mod_{\on{reg}})^{\geq r},
\] 
or equivalently, that $\Gamma^{\Hecke}(\sG\star\on{ind}^{\Hecke_{\fz}}(\sF_{i,j})\otimes_{\Op_{\check{G}}^{\reg}}\sH _{i,j})$ is eventually coconnective.

As $\QCoh(\Op_{\check{G}}^{\reg})^{\heartsuit}$ is compactly generated, we can immediately reduce to the case of $\sH _{i,j}$ coherent. And as $\Op_{\check{G}}^{\reg}$ is an infinite-dimensional affine space, any coherent sheaf is perfect. Therefore, to prove eventual coconnectivity, the case of $\sH _{i,j}$ coherent further reduces to the case where $\sH _{i,j}$ is equivalent to the structure sheaf $\sO_{\Op_{\check{G}}^{\reg}}$.

So we just need to show that
\begin{align*}
\Gamma^{\Hecke}(\sG\star\on{ind}^{\Hecke_{\fz}}(\sF_{i,j}))&\cong\Gamma^{\Hecke}(\on{ind}^{\Hecke_{\fz}}(\sG\star\sF_{i,j}))\\
&\cong\Gamma^{\IndCoh}(\Gr_G,\sG\star\sF_{i,j})
\end{align*}
is eventually coconnective. Applying Lemma \ref{boundedconvolution} again, we see that $\sG\star\sF_{i,j}\in D_{\crit}(\Gr_G)^{\geq r}$. By Corollary 7.15.2 of \cite{raskinbeiber}, $\Gamma^{\IndCoh}(\Gr_G,-)$ is t-exact, giving the desired statement.
\end{proof}

\subsection{$(K_{x,r+},\circ)$-invariants}

It will turn out that both $D^{\Hecke_{\fz}}_{\crit}(\Gr_G)$ and $\widehat{\fg}_{\crit}\mod_{\on{reg,naive}}$ are purely of depth zero, i.e., are equal to their depth $\leq 0$ truncations. By Theorem \ref{unrefined}, this is equivalent to saying that for all $x$ and $r$, their $(K_{x,r+},\circ)$-invariant categories are trivial.

The following lemma will help show this in both cases.

\begin{lemma}\label{tensordepthzero}
\begin{enumerate}
\item Let $\sC$ be a $G((t))$-category purely of depth zero, with a right action of a monoidal category $\sA$ that commutes with the $G((t))$-action. Let $\sB$ be an $\sA$-module category. Then the category
\[
\sC\underset{\sA}{\otimes} \sB
\]
is purely of depth zero.
\item Let $\sC$ be a $G((t))$-category purely of depth zero, with a right coaction of a comonoidal category $\sA$ that commutes with the $G((t))$-action. Let $\sB$ be an $\sA$-comodule category. Then the category
\[
\sC\overset{\sA}{\otimes}\sB
\]
is purely of depth zero.
\end{enumerate}
\end{lemma}
\begin{proof}
Let us first treat the case of $\sA$ monoidal. We would like to show that
\[
(\sC\underset{\sA}{\otimes} \sB)^{K_{x,r+},\circ}
\]
is trivial for any choices of $x$ and $r$. But the operation of $(K_{x,r+},\circ)$-invariants commutes with tensor products, so we have
\begin{align*}
(\sC\underset{\sA}{\otimes} \sB)^{K_{x,r+},\circ}&\cong \sC^{K_{x,r+},\circ}\underset{\sA}{\otimes} \sB\\
&\cong 0\underset{\sA}{\otimes} \sB\\
&\cong 0.
\end{align*}
as desired.

The case of $\sA$ comonoidal proceeds similarly, as taking $(K_{x,r+},\circ)$-invariants also commutes with cotensor products.
\end{proof}

\begin{theorem}\label{grdepthzero}
The category $D^{\Hecke_{\fz}}_{\crit}(\Gr_G)$ is purely of depth zero.
\end{theorem}

\begin{proof}
Start by noting that $D(G((t))/K_{0,0+})$ is tautologically $K_{0,0+}$-generated and hence purely of depth zero. As $K_{0,0+}$ is a normal subgroup of $K_{0,0}\cong G[[t]]$, with quotient $G$, we have
\[
D(\Gr_G)\cong D(G((t))/G[[t]])\cong D(G((t))/K_{0,0+})^G\cong D(G((t))/K_{0,0+})\underset{D(G)}{\otimes}\on{Vect}.
\]

Thus, by the first part of Lemma \ref{tensordepthzero}, $D(\Gr_G)$ is purely of depth zero. As
\[
D^{\Hecke_{\fz}}_{\crit}(\Gr_G)\cong D(\Gr_G)\underset{\on{Rep}(\check{G})}{\otimes}\QCoh(\Op_{\check{G}}^{\on{reg}}),
\]
another application of Lemma \ref{tensordepthzero} shows that $D^{\Hecke_{\fz}}_{\crit}(\Gr_G)$ is purely of depth zero, as desired.
\end{proof}

Now we treat the case of $\widehat{\fg}_{\crit}\mod_{\on{reg,naive}}$. We will need the following special case of Theorem 4.3 of \cite{dhy}, which computes the depth filtration on $\widehat{\fg}_{\crit}\mod$.

\begin{theorem}[\cite{dhy} Theorem 4.3, Case $r=0$]
There is an equivalence of subcategories
\[
\widehat{\fg}_{\crit}\mod_{\leq 0}\cong\widehat{\fg}_{\crit}\mod\overset{\IndCoh^*(\Op_{\check{\fg}})}{\otimes}\IndCoh^*(\widehat{\Op_{\check{\fg}}^{\leq 0}}),
\]
where $\widehat{\Op_{\check{\fg}}^{\leq 0}}$ is the formal completion of $\Op_{\check{\fg}}^{\leq 0}$ inside $\Op_{\check{\fg}}$.
\end{theorem}

Properly equipped, let us proceed to the promised theorem.

\begin{theorem}\label{kmdepthzero}
The category $\widehat{\fg}_{\crit}\mod_{\on{reg,naive}}$ is purely of depth zero.
\end{theorem}

\begin{proof}
Note that $\IndCoh^*(\widehat{\Op_{\check{\fg}}^{\leq 0}})$ has a natural coalgebra structure. Furthermore, as $\Op_{\check{\fg}}^{\on{reg}}\subseteq\widehat{\Op_{\check{\fg}}^{\leq 0}}, \IndCoh^*(\Op_{\check{\fg}}^{\on{reg}})$ is naturally a comodule over $\IndCoh^*(\widehat{\Op_{\check{\fg}}^{\leq 0}})$. Thus, we can write:
\begin{align*}
&\widehat{\fg}_{\crit}\mod_{\on{reg,naive}}\\
\cong&\widehat{\fg}_{\crit}\mod\overset{\IndCoh^*(\Op_{\check{\fg}})}{\otimes}\IndCoh^*(\Op_{\check{\fg}}^{\on{reg}})\\
\cong&(\widehat{\fg}_{\crit}\mod\overset{\IndCoh^*(\Op_{\check{\fg}})}{\otimes}\IndCoh^*(\widehat{\Op_{\check{\fg}}^{\leq 0}}))\overset{\IndCoh^*(\widehat{\Op_{\check{\fg}}^{\leq 0}})}{\otimes}\IndCoh^*(\Op_{\check{\fg}}^{\on{reg}})\\
\cong&\widehat{\fg}_{\crit}\mod_{\leq 0}\overset{\IndCoh^*(\widehat{\Op_{\check{\fg}}^{\leq 0}})}{\otimes}\IndCoh^*(\Op_{\check{\fg}}^{\on{reg}}).
\end{align*}

The second part of Lemma \ref{tensordepthzero} now implies the desired statement.

\end{proof}

\section{Generation}

\subsection{Reduction step}

As shorthand, we say that a (not necessarily fully faithful) functor $f:\sC\rightarrow \sD$ between categories generates under colimits if the essential image of $f$ generates $\sD$ under colimits.

\begin{lemma}\label{quotientcriterion}
Let $\sC$ and $\sD$ be two categories and let $f:\sC\rightarrow \sD$ be a functor between them. Then $f$ generates under colimits if and only if the quotient category, i.e., the colimit of the pushout diagram
\begin{equation*}
\begin{tikzcd}
\sC \arrow[r] \arrow[d] & \sD\\
0, &
\end{tikzcd}
\end{equation*}
is trivial.
\end{lemma}

\begin{proof}
It is well known (see e.g. Lemma I.1.5.4.3 of \cite{GaRo}) that $f$ generates under colimits iff its (not necessarily continuous) right adjoint $f^R$ is conservative. Conservativeness is equivalent to $f^R$ having trivial kernel, and by Corollary 5.5.3.4 of \cite{HTT}, the kernel of $f^R$ can be identified with the quotient of $f$, giving the desired statement.
\end{proof}

We will deduce Lemma \ref{generation} from the following auxillary statement.

\begin{lemma}\label{generationreductionstep}
Let $\sC,\sD$ be $G((t))$-categories and let $f:\sC\rightarrow \sD$ be a $G((t))$-equivariant functor. Then $f$ generates under colimits if and only if the functors
\[
f^{K_{x,r+},\circ}:\sC^{K_{x,r+},\circ}\rightarrow \sD^{K_{x,r+},\circ},
\]
for all choices of $r>0$ and $x\in X_*(T)\otimes\mathbb{R}$, as well as the functor
\[
f^{I_0}:\sC^{I_0}\rightarrow \sD^{I_0},
\]
all generate under colimits.
\end{lemma}

\begin{proof}
First we show the ``only if" direction. Let $\sE$ be the quotient of $\sD$ by $\sC$. By Lemma \ref{quotientcriterion}, $\sE$ must be trivial.

As the functor of $I_0$-invariants commutes with colimits, the quotient of $\sD^{I_0}$ by $\sC^{I_0}$ is $\sE^{I_0}$. Similarly, the quotient of $\sD^{K_{x,r+},\circ}$ by $\sC^{K_{x,r+},\circ}$ is $\sE^{K_{x,r+},\circ}$. Since $\sE$ is trivial, these quotients must also be trivial, and hence the corresponding functors must generate under colimits, as desired.

Now let us show the ``if" direction. Reversing the logic, we see that we have
\[
\sE^{I_0}\cong 0
\]
and, for all $r>0$ and $x\in X_*(T)\otimes\mathbb{R}$,
\[
\sE^{K_{x,r+},\circ}\cong 0,
\]
and we need to show that $\sE\cong 0$. But this is exactly the statement of Corollary \ref{vanishingcriterion}.
\end{proof}

\subsection{The category $\widehat{\fg}_{\crit}\widetilde{\mod}_{\on{reg,naive}}$}

Recall that we defined $\widehat{\fg}_{\crit}\widetilde{\mod}_{\on{reg,naive}}$ to be the full subcategory of $\widehat{\fg}_{\crit}\mod_{\regnaive}$ generated by $\widehat{\fg}_{\crit}\mod_{\regnaive}^+$ under colimits. We will need a few properties of this category which were proven in \cite{raskinbeiber}. As the proofs would take us slightly afield, we do not reproduce them here. 

\begin{lemma}[\cite{raskinbeiber}, Lemma 9.2.1]
There is a (necessarily unique) structure of $G((t))$-category on $\widehat{\fg}_{\crit}\widetilde{\mod}_{\regnaive}$ compatible with the inclusion into $\widehat{\fg}_{\crit}\mod_{\regnaive}$.
\end{lemma}

\begin{lemma}[\cite{raskinbeiber}, Corollary 9.2.3]\label{generationfactor}
$\Gamma^{\on{Hecke,naive}}$ factors through $\widehat{\fg}_{\crit}\widetilde{\mod}_{\regnaive}$.
\end{lemma}

\begin{lemma}[\cite{raskinbeiber}, Corollary 9.2.4]\label{tildeinvariants}
Let $K\subseteq G[[t]]$ be a prounipotent group scheme. Then $\widehat{\fg}_{\crit}\widetilde{\mod}_{\regnaive}^K$ is the subcategory of $\widehat{\fg}_{\crit}\mod_{\regnaive}^K$ generated under colimits by $\widehat{\fg}_{\crit}\mod_{\regnaive}^{K,+}$.
\end{lemma}

We will also need one new lemma.

\begin{lemma}\label{kmtildedepthzero}
Choose $x\in X_*(T)\otimes\mathbb{R}$ and $r>0$. Then
\[
\widehat{\fg}_{\crit}\widetilde{\mod}_{\regnaive}^{K_{x,r+},\circ}\cong 0.
\]
\end{lemma}

\begin{proof}
Consider the map
\[
\widehat{\fg}_{\crit}\widetilde{\mod}_{\regnaive}^{K_{x,r+},\circ}\rightarrow\widehat{\fg}_{\crit}\mod_{\regnaive}^{K_{x,r+},\circ}.
\]

We will compute its kernel $\sA$ in two different ways. (The kernel of a map $\sC\rightarrow \sD$ is defined to be the pullback of $\sC$ along the map $0\rightarrow \sD$.) First, since the operation of taking $(K_{x,r+},\circ)$-invariants commutes with limits, we have
\[
\sA\cong(\on{ker}\widehat{\fg}_{\crit}\widetilde{\mod}_{\regnaive}\rightarrow\widehat{\fg}_{\crit}\mod_{\regnaive})^{K_{x,r+},\circ}.
\]

As the map $\widehat{\fg}_{\crit}\widetilde{\mod}_{\regnaive}\rightarrow\widehat{\fg}_{\crit}\mod_{\regnaive}$ is the inclusion of a full subcategory, it has trivial kernel, and so $\sA$ must also be trivial.

On the other hand, Theorem \ref{kmdepthzero} shows that $\widehat{\fg}_{\crit}\mod_{\regnaive}^{K_{x,r+},\circ}\cong 0$. Thus, $\sA$ can also be identified with $\widehat{\fg}_{\crit}\widetilde{\mod}_{\regnaive}^{K_{x,r+},\circ}$. Comparing these two computations, we get the desired equivalence.
\end{proof}

Now we are prepared to prove Lemma \ref{generation}.

\subsection{Proof of Lemma \ref{generation}}

\begin{proof}
By Lemma \ref{generationfactor}, we know that the image of $\Gamma^{\on{Hecke,naive}}$ lies in $\widehat{\fg}_{\crit}\widetilde{\mod}_{\regnaive}$. It remains to show the generation under colimits.

For this, we apply Lemma \ref{generationreductionstep}. By Theorem \ref{grdepthzero} and Lemma \ref{kmtildedepthzero}, we have
\[
D^{\Hecke_{\fz}}_{\crit}(\Gr_G)^{K_{x,r+},\circ}\cong\widehat{\fg}_{\crit}\widetilde{\mod}_{\regnaive}^{K_{x,r+},\circ}\cong 0.
\]

So the only statement remaining to be checked is that the functor
\[
D^{\Hecke_{\fz}}_{\crit}(\Gr_G)^{I_0}\rightarrow\widehat{\fg}_{\crit}\widetilde{\mod}_{\regnaive}^{I_0}
\]
generates under colimits. This follows from the combination of Theorem \ref{i0invariants} and Lemma \ref{tildeinvariants}.
\end{proof}

\section{Exactness}\label{exactnesssection}

The goal of this section is to prove Lemma \ref{exactness}. We start by formulating a general criterion for a $G((t))$-equivariant functor to be left t-exact. 

\subsection{Reduction step}

We will use an analogue of Lemma \ref{generationreductionstep}. To state it, let us note that for a $G((t))$-category $\sC$ with a compatible t-structure, Lemma \ref{compatibletensor} constructs a t-structure on $\sC^{K_{x,r+},\circ}$.

\begin{lemma}\label{exactnessreductionstep}
Let $\sC,\sD$ be $G((t))$-categories equipped with compatible t-structures and let $f:\sC\rightarrow \sD$ be a $G((t))$-equivariant functor. Then $f$ is left t-exact if and only if the functors
\[
f^{K_{x,r+},\circ}:\sC^{K_{x,r+},\circ}\rightarrow \sD^{K_{x,r+},\circ},
\]
for all choices of $r>0$ and $x\in X_*(T)\otimes\mathbb{R}$, as well as the functor
\[
f^{I_0}:\sC^{I_0}\rightarrow \sD^{I_0},
\]
are all left t-exact.
\end{lemma}

\begin{proof}

Let us quickly dispatch of the only if case. The left t-exactness of $f^{I_0}$ is clear, and the left t-exactness of $f^{K_{x,r+},\circ}$ follows from the commutativity of the diagram:
\begin{equation*}
\begin{tikzcd}
\sC^{K_{x,r+},\circ}\arrow[r]\arrow[d] & \arrow[d] \sD^{K_{x,r+},\circ}\\
\sC^{K_{x,r+}}\otimes D((K_{x,r}/K_{x,r+})^{*,\circ})\arrow[r] & \sD^{K_{x,r+}}\otimes D((K_{x,r}/K_{x,r+})^{*,\circ}).
\end{tikzcd}
\end{equation*}
Indeed, the vertical arrows are conservative and t-exact by Lemma \ref{compatibletensor}, and the bottom arrow is left t-exact by Proposition \ref{cgtensor1}.

Now we treat the (much harder) if direction. We will show that for any choice of $x\in X_*(T)\otimes\mathbb{R}$ and $r\geq 0$, the functor $f^{K_{x,r+}}:\sC^{K_{x,r+}}\rightarrow \sD^{K_{x,r+}}$ is left t-exact. This implies that $f$ is left t-exact. Indeed, for any object $c\in \sC^{\geq 0}$, we have
\begin{align*}
f(c)&\cong\on{colim}\on{Oblv}(\on{Av}_*^{K_{0,r+}}f(c))\\
&\cong\on{colim}\on{Oblv}(f^{K_{0,r+}}(\on{Av}_*^{K_{0,r+}}c))
\end{align*}
and the result follows from left t-exactness of $\on{Av}_*$.

Our proof will be by induction. First we deal with the base case.

\subsubsection{Depth zero}\label{exactnessdepthzero}

Assume $r=0$. We will emulate the proof of Theorem \ref{unrefineddepthzero}. Recall from that proof that any sufficiently small $\epsilon>0$ gives the same subgroup $K_{x+\epsilon\check{\rho},0+}$, and that $K_{x+\epsilon\check{\rho},0+}$ is a conjugate of $I_0$. In particular, the functor $f^{K_{x+\epsilon\check{\rho},0+}}$ is left t-exact.

Let $N_x$ denote the subgroup $(K_{x,0+}(P_x\cap N((t))))/K_{x,0+}$ of $L_x$. As in the proof of Theorem \ref{unrefineddepthzero}, $N_x$ is a maximal unipotent subgroup of $L_x$. We have an equivalence $\sC^{K_{x+\epsilon\check{\rho},0+}}\cong (\sC^{K_{x,0+}})^{N_x}$. The left t-exactness of $f^{K_{x,0+}}$ now follows from the following purely finite-dimensional lemma (taking $H$ to be $L_x$ and $g$ to be $f^{K_{x,0+}}$):

\begin{lemma}
Let $H$ be a reductive algebraic group over $k$, with Borel, Cartan, etc. subgroups denoted by $B_H, T_H$, etc. Let $\sA$ and $\sB$ be two $H$-categories equipped with compatible t-structures, and let $g:\sA\rightarrow \sB$ be a $H$-equivariant functor. If $g^{N_H}:\sA^{N_H}\rightarrow \sB^{N_H}$ is left t-exact, then $g$ is left t-exact.
\end{lemma}

\begin{proof}
We argue via a sequence of commutative diagrams. They will implicitly introduce some functors $g_i$, which are all induced by $g$. First, examine the diagram
\begin{equation*}
\begin{tikzcd}
(\sA^{N_H})^{T_H,w} \arrow[r, "g_2"] \arrow[d] & (\sB^{N_H})^{T_H,w} \arrow[d]\\
\sA^{N_H} \arrow[r, "g_1"] & \sB^{N_H}.
\end{tikzcd}
\end{equation*}
The vertical arrows are conservative and t-exact by part (4) of Lemma \ref{compatiblelemma}. As $g_1$ is left t-exact by assumption, we deduce that $g_2$ is left t-exact.

Next, we have:
\begin{equation*}
\begin{tikzcd}
(\sA^{N_H})^{T_H,w}\underset{\QCoh(\ft_{\fh}^*)}{\otimes} \QCoh(\ft_{\fh}^{*,\circ})\arrow[r, "g_3"] \arrow[d] & (\sB^{N_H})^{T_H,w}\underset{\QCoh(\ft_{\fh}^*)}{\otimes}\QCoh(\ft_{\fh}^{*,\circ}) \arrow[d]\\
(\sA^{N_H})^{T_H,w} \arrow[r, "g_2"] & (\sB^{N_H})^{T_H,w}.
\end{tikzcd}
\end{equation*}
By Lemma \ref{flattstructure}, the vertical arrows are conservative and t-exact. From the previous diagram, we learned that $g_2$ is left t-exact. Hence, $g_3$ is left t-exact.

Now we use Lemma \ref{finitebeiber}. Compatibility of $\Gamma$ with $H$-equivariant functors gives commutativity of the diagram
\begin{equation*}
\begin{tikzcd}
(\sA^{N_H})^{T_H,w}\underset{\QCoh(\ft_{\fh}^*)}{\otimes}\QCoh(\ft_{\fh}^{*,\circ})\arrow[r, "g_3"] \arrow[d, "\Gamma"] & (\sB^{N_H})^{T_H,w}\underset{\QCoh(\ft_{\fh}^*)}{\otimes}\QCoh(\ft_{\fh}^{*,\circ}) \arrow[d, "\Gamma"]\\
\sA^{H,w}\underset{\QCoh(\ft_{\fh}^*//W_H)}{\otimes}\QCoh(\ft_{\fh}^{*,\circ})\arrow[r, "g_4"] & \sB^{H,w}\underset{\QCoh(\ft_{\fh}^*//W_H)}{\otimes}\QCoh(\ft_{\fh}^{*,\circ}).
\end{tikzcd}
\end{equation*}
Let $x$ be an object in $(\sA^{H,w}\otimes_{\QCoh(\ft_{\fh}^*//W_H)}\QCoh(\ft_{\fh}^{*,\circ}))^{\geq 0}$. Then if we take $y$ to be $\tau^{\geq 0}\on{Loc} x$, we have
\begin{align*}
\Gamma(y)&\cong\Gamma\circ\tau^{\geq 0}\circ\on{Loc}(x)\\
&\cong\tau^{\geq 0}\circ\Gamma\circ\on{Loc}(x)\\
&\cong\tau^{\geq 0} x\\
&\cong x.
\end{align*}

We can use this to calculate that $g_4(x)\cong g_4(\Gamma(y))\cong\Gamma(g_3(y))$. As $g_3$ and $\Gamma$ are both left t-exact, we see that
\[
g_4(x)\in (\sB^{H,w}\underset{\QCoh(\ft_{\fh}^*//W_H)}{\otimes}\QCoh(\ft_{\fh}^{*,\circ}))^{\geq 0}
\]
so $g_4$ is also left t-exact, as desired.

The next diagram we consider is:
\begin{equation*}
\begin{tikzcd}
\sA^{H,w}\arrow[r, "g_5"] \arrow[d] & \sB^{H,w}\arrow[d]\\
\sA^{H,w}\underset{\QCoh(\ft_{\fh}^*)}{\otimes}\QCoh(\ft_{\fh}//W_H)\arrow[r, "g_4"] & \sB^{H,w}\underset{\QCoh(\ft_{\fh}^*)}{\otimes} \QCoh(\ft_{\fh}//W_H).
\end{tikzcd}
\end{equation*}
Again, we just showed that $g_4$ is left t-exact. We showed in the proof of Corollary \ref{n-generation} that the vertical arrows are conservative, and they are t-exact by Lemma \ref{flattstructure}. So $g_5$ is left t-exact.

The final diagram:
\begin{equation*}
\begin{tikzcd}
\sA \arrow[r, "g"] \arrow[d] & \sB \arrow[d]\\
\sA^{H,w} \arrow[r, "g_5"] & \sB^{H,w}.
\end{tikzcd}
\end{equation*}
The vertical arrows are conservative and t-exact by Lemma \ref{compatiblelemma}. Thus the left t-exactness of $g_5$ implies the left t-exactness of $g:\sA\rightarrow \sB$, as desired.
\end{proof}

\subsubsection{Positive depth}

Now we treat the inductive step. Assume that for all $s < r$ and $x\in X_*(T)\otimes\mathbb{R}$, we know that $f^{K_{x,s+}}$ is left t-exact. Now, fixing a choice of $x$, we want to show that $f^{K_{x,r+}}$ is left t-exact.

There are two natural adjoint pairs
\[
i_{\sC,*}: \sC^{K_{x,r+},us}\rightleftarrows \sC^{K_{x,r+}}: i_{\sC}^!
\]
and
\[
j_{\sC}^!: \sC^{K_{x,r+}}\rightleftarrows \sC^{K_{x,r+},\circ}: j_{\sC,*}.
\]
For any object $c\in \sC^{K_{x,r+}}$, we have an exact triangle
\[
i_{\sC,*}i_{\sC}^!c\rightarrow c\rightarrow j_{\sC,*}j_{\sC}^!c\rightarrow i_{\sC,*}i_{\sC}^!c[1].
\]

We have analogous functors $i_{\sD,*}$, etc. As $f$ is $G((t))$-equivariant, $f^{K_{x,r+}}$ intertwines the functors associated to $\sC$ and $\sD$. Thus, we get an exact triangle
\[
i_{\sD,*}f^{K_{x,r+},us}(i_{\sC}^!c)\rightarrow f^{K_{x,r+}}(c)\rightarrow j_{\sD,*}f^{K_{x,r+},\circ}(j_{\sC}^!c)\rightarrow i_{\sD,*}f^{K_{x,r+},us}(i_{\sC}^!c)[1].
\]
In particular, to show that $f$ is left t-exact, it suffices to show the same for $i_{\sD,*}\circ f\circ i_{\sC}^!$ and $j_{\sD,*}\circ f\circ j_{\sC}^!$. By Lemma \ref{exactsequence}, the functors $i_*$ and $j^!$ are t-exact, and so their right adjoints $i^!$ and $j_*$ must be left t-exact. Thus, it suffices to show that $f^{K_{x,r+},\circ}$ and $f^{K_{x,r+},us}$ are left t-exact. We assumed that $f^{K_{x,r+},\circ}$ is left t-exact, so we can restrict our attention to the case of $f^{K_{x,r+},us}$.

More generally, for $V$ a quasi-projective variety mapping to $(K_{x,r}/K_{x,r+})^*$, we can define a category
\[
\sC^{K_{x,r+},V}\cong \sC^{K_{x,r+}}\underset{D((K_{x,r}/K_{x,r+})^*)}{\otimes} D(V)
\]
and a functor
\[
f^{K_{x,r+},V}:\sC^{K_{x,r+},V}\rightarrow \sD^{K_{x,r+},V}
\]
via tensor product. Let $U\subseteq V$ be an open subvariety with complement $Z$. The above argument then shows that if $f^{K_{x,r+},U}$ and $f^{K_{x,r+},Z}$ are left t-exact, then so is $f^{K_{x,r+},V}$.

By Lemma \ref{finiteorbits}, $(K_{x,r}/K_{x,r+})^{*,us}$ is the union of a finite number of $L_x$-orbits. We therefore can reduce the left t-exactness of $f^{K_{x,r+},us}$ to the left t-exactness of the functors $f^{K_{x,r+},O}$ for $O$ an $L_x$-orbit in $(K_{x,r}/K_{x,r+})^{*,us}$. Fix such an orbit $O$.

Now, Lemma \ref{geometricinput} tells us that there exists a point $p\in O$ and $y\in X_*(T)\otimes\mathbb{R}$ such that $K_{x,r+}\subseteq K_{y,r}\subseteq K_{x,r}$ and $p\in(K_{y,r}/K_{x,r+})^{\perp}$. Note that we can identify $\sC^{K_{x,r+},(K_{y,r}/K_{x,r+})^{\perp}}$ with $\sC^{K_{y,r}}$. Indeed, we compute:
\begin{align*}
\sC^{K_{x,r+}}\underset{D((K_{x,r}/K_{x,r+})^*)}{\otimes} D((K_{y,r}/K_{x,r+})^{\perp})&\cong \sC^{K_{x,r+}}\underset{D((K_{y,r}/K_{x,r+})^*)}{\otimes} \on{Vect}\\
&\cong \sC^{K_{x,r+}}\underset{D(K_{y,r}/K_{x,r+})}{\otimes} \on{Vect}\\
&\cong \sC^{K_{y,r}}.
\end{align*}
This leads to a commutative diagram:
\begin{equation*}
\begin{tikzcd}
\sC^{K_{x,r+},p} \arrow[r] \arrow[d] & \sD^{K_{x,r+},p} \arrow[d]\\
\sC^{K_{y,r}}\arrow[r] & \sD^{K_{y,r}}.
\end{tikzcd}
\end{equation*}

As $K_{y,r}=K_{y,(r-\epsilon)+}$ for sufficiently small $\epsilon>0$, the inductive hypothesis implies that the bottom arrow is left t-exact. The vertical arrows are fully faithful and t-exact by Lemma \ref{exactsequence}. Therefore, the top arrow, which can be identified with $f^{K_{x,r+},p}$, is necessarily left t-exact.

Let $S_p$ be the stabilizer group of $p$ with respect to the action of $P_x/K_{x,r+}$ on $O$. Then we claim that there is a canonical t-exact equivalence
\[
D(P_x/K_{x,r+})\underset{S_p}{\otimes} \sC^{K_{x,r+},p}\cong \sC^{K_{x,r+},O}.
\]
Indeed, by Lemma \ref{orbitstabilizerequivalence}, for any $P_x/K_{x,r+}$-category $\sE$, we have a t-exact equivalence
\[
D(P_x/K_{x,r+})\underset{S_p}{\otimes} (\sE\underset{K_{x,r}/K_{x,r+}}{\otimes} \QCoh(p))\cong \sE\underset{K_{x,r}/K_{x,r+}}{\otimes} \QCoh(O).
\]

To finish, consider the following commutative square.
\begin{equation*}
\begin{tikzcd}
\sC^{K_{x,r+},O} \arrow[r] \arrow[d] & \sD^{K_{x,r+},O} \arrow[d] \\
\sC^{K_{x,r+},p}\otimes D(P_x/K_{x,r+})\arrow[r] & \sD^{K_{x,r+},p}\otimes D(P_x/K_{x,r+})
\end{tikzcd}
\end{equation*}

The vertical arrows are conservative and t-exact by Lemma \ref{compatibletensor} and the bottom arrow is left t-exact by Proposition \ref{cgtensor1}. Thus the top arrow, which is given by $f^{K_{x,r+},O}$, is left t-exact, as desired.

\end{proof}

In the above proof, we used the following lemma to relate $\sC^{K_{x,r+},O}$ and $\sC^{K_{x,r+},p}$:

\begin{lemma}\label{orbitstabilizerequivalence}
Let $B$ be a finite type algebraic group and let $A\subseteq B$ be a normal subgroup which is isomorphic to the algebraic group underlying a $k$-vector space (which, by abuse of notation, we also denote by $A$). Let $\sE$ be a category acted on by $B$ and let $p\in A^*$ be an additive character of $A$. If $O$ denotes the $B$-orbit of $p$ inside $A^*$ and $S_p\subseteq B$ denotes the stabilizer of $p$, then we have a canonical equivalence
\[
\sE\underset{A}{\otimes} D(O)\cong (D(p)\underset{A}{\otimes} \sE)\underset{S_p}{\otimes} D(B).
\]
Furthermore, if $\sE$ is equipped with a t-structure compatible with the $B$ action, then the above equivalence is t-exact.
\end{lemma}

\begin{proof}
It suffices to consider the universal case of $\sE\cong D(B)$. We have a diagram
\begin{equation*}
\begin{tikzcd}
D(B)\underset{A}{\otimes} D(O) \arrow[d] & (D(p)\underset{A}{\otimes} D(B))\underset{S_p}{\otimes} D(B)\arrow[d]\\
D(B)\otimes D(O) \arrow[r] & D(B)\underset{S_p}{\otimes} D(B)
\end{tikzcd}
\end{equation*}
where the bottom arrow is the equivalence induced by the isomorphism
\[
B\times^{S_p}B\cong B\times O
\]
sending $(b_1,b_2)$ to $(b_1b_2, (b_2)^{-1}\cdot p)$. The vertical arrows are fully faithful and it is straightforward to check that their essential images agree, giving us the desired equivalence. As every arrow in the above diagram is t-exact, the t-exactness of said equivalence also follows.
\end{proof}

\subsection{Proof of Lemma \ref{exactness}}

Now that we have Lemma \ref{exactnessreductionstep}, it will not take much to deduce Lemma \ref{exactness}.

\begin{proof}

Combining Theorem \ref{righttexact} with Theorem \ref{eventuallyequivalence}, we see that $\Gamma^{\Heckenaive}$ is right t-exact. So it suffices to prove left t-exactness, which we do via Lemma \ref{exactnessreductionstep}.

We need to show left t-exactness of the maps $(\Gamma^{\on{Hecke,naive}})^{I_0}$ and, for any $x\in X_*(T)\otimes\mathbb{R}$ and $r > 0$, $(\Gamma^{\on{Hecke,naive}})^{K_{x,r+},\circ}$. The case of $(K_{x,r+},\circ)$-invariants is trivial because
\[
D^{\Hecke_{\fz}}_{\crit}(\Gr_G)^{K_{x,r+},\circ}\cong 0
\]
by Theorem \ref{grdepthzero}. On the other hand, for $I_0$-invariants, the desired left t-exactness follows from Theorem \ref{i0invariants}. 

\end{proof}

\section{Boundedness}

This section is mainly concerned with Lemma \ref{boundedness}. The proofs will be similar to those in Section \ref{exactnesssection}. First we will prove two helpful lemmas.

\subsection{Two auxilliary lemmas}

\subsubsection{Boundedness of convolution}

The following lemma shows (under appropriate hypotheses) that convolution is left t-exact up to shift. We actually will not use it in this section (though it was used earlier in this paper.) However, it explains why the statement of Lemma \ref{boundednessreductionstep} is reasonable.

\begin{lemma}\label{boundedconvolution}
Let $K\subseteq G((t))$ be a compact open subgroup and let $\sG$ be a compact object of $D(G((t))/K)$. Then there is some integer $r$ such that, for any $G((t))$-category $\sC$ with a compatible t-structure, the convolution functor
\[
\sG\star-:\sC\rightarrow \sC
\]
sends $\sC^{\geq 0}$ to $\sC^{\geq r}$.
\end{lemma}

\begin{proof}
This is essentially Lemma 9.2.2 of \cite{raskinbeiber}. We reproduce the proof here.

By compactness of $\sG$, it is supported on some (finite type) subscheme $S\subseteq G((t))/K$. Then $\sG$ has a finite resolution by inductions $\on{ind}(\sG_i)$ of bounded below ind-coherent sheaves $\sG_i\in\IndCoh(S)^+$. So we can assume that $\sG$ is of the form $\on{ind}(\sG')$. Without loss of generality, we assume that $\sG'\in\IndCoh(S)^{\geq 0}$. 

The functors of strong convolution with $\on{ind}(\sG')$ and weak convolution with $\sG'$ are canonically isomorphic. By Lemma 10.16.1 of \cite{semiinf}, weak convolution with $\sG'$ is left exact, as desired.
\end{proof}

\subsubsection{Boundedness and invariant categories}

The other Lemma we prove will be used below in the proof of Lemma \ref{boundednessreductionstep}. It will allow us ``deequivariantize" boundedness statements on categories of weak invariants.

\begin{lemma}\label{boundedinvariants}
Let $H$ be a finite type algebraic group and let $\sC$ be a category with a weak $H$ action and a compatible t-structure. Assume that we are given a functor $f:\sC\rightarrow \sD$ with $\sD$ a category with a t-structure. Then $f$ is left t-exact up to shift if and only if $f\circ\on{Oblv}:\sC^{H,w}\rightarrow \sD$ is left t-exact up to shift.
\end{lemma}

\begin{proof}
The only if case follows immediately from Lemma \ref{compatiblelemma}. Let us treat the if case. Taking an appropriate shift, we may assume that $f$ is left t-exact.

Let $\sF$ be an object in $\sC^{\geq 0}$. Identifying $\on{Oblv}\circ\on{Av}^{H,w}$ with the convolution functor $\sO_H\star-:\sC\rightarrow \sC$, we have
\[
f(\sO_H\star\sF)\cong f(\on{Oblv}(\on{Av}^{H,w}\sF))\in \sD^{\geq 0},
\]
using Lemma \ref{compatiblelemma} again. It follows that for any projective coherent sheaf $\sM$ on $H$, we again have
\[
f(\sM\star\sF)\in \sD^{\geq 0}.
\]

As $H$ is smooth, any coherent sheaf on $H$ has a finite resolution by projective coherent sheaves. In particular, the skyscraper sheaf at the identity (which we denote by $1_H$) admits such a resolution. Thus, we can conclude that
\[
f(\sF)\cong f(1_H\star\sF)\in \sD^+
\]
as desired.
\end{proof}

\subsection{Reduction step}

Before stating the key lemma, let us introduce a convenient notational shorthand. Assume we have a functor $f:\sC\rightarrow \sD$, where $\sC$ is a $G((t))$-category and $\sD$ is any category. For $K\subseteq G((t))$ a compact open subgroup and $\sG$ a compact object of $D(G((t))/K)$, we will write $f_{\sG}|_{K}$ for the functor $f(\sG\star-): \sC^K\rightarrow \sD$. When $\sG$ is the skyscraper sheaf at the identity, $f_{\sG}|_K$ can be identified with the composition $\sC^K\rightarrow \sC\rightarrow \sD$ and will be denoted simply by $f|_K$.

\begin{lemma}\label{boundednessreductionstep}
Let $\sC$ be a $G((t))$-category with a compatible t-structure, and let $\sD$ be a category with a t-structure. Assume that $f:\sC\rightarrow \sD$ is a functor satisfying the following two properties.
\begin{itemize}
\item For any compact object $\sG$ in $D(G((t))/I_0)$, $f_{\sG}|_{I_0}$ is left t-exact up to shift.
\item For any point $x\in\mathbb{X}_*(T)\otimes\mathbb{R}$, real number $r>0$, and compact object $\sG\in D(G((t))/K_{x,r+})$, the restriction of $f_{\sG}|_{K_{x,r+}}$ to $\sC^{K_{x,r+},\circ}$ is left t-exact up to shift.
\end{itemize}
Then for any compact open subgroup $K\subseteq G((t))$ and any compact object $\sG\subseteq D(G((t))/K)$, $f_{\sG}|_K$ is left t-exact up to shift.
\end{lemma}

\begin{proof}

As every compact open subgroup contains some $K_{x,r+}$, it suffices to show that each $f_{\sG}|_{K_{x,r+}}$ is left t-exact up to shift, for any $x\in X_*(T)\otimes\mathbb{R}$ and $r\geq 0$. We prove this by induction. At each step of the induction, it suffices to show left t-exactness up to shift of $f|_{K_{x,r+}}$. Indeed, the case of $f_{\sG}|_{K_{x,r+}}$ follows by replacing $f$ with $f_{\sG}$.

Again, we will treat the base case and inductive steps separately.

\subsubsection{Depth zero}

For the base case, we use the same setup as in Section \ref{exactnessdepthzero}. In that notation, we need to show the left t-exactness up to shift of $(\sC^{K_{x,0+}})^{N_x}\rightarrow \sD$. Once again, taking $H$ to be $L_x$ and $g$ to be $f|_{K_{x,0+}}$,this follows from a finite-dimensional lemma.

\begin{lemma}
Let $H$ be a reductive algebraic group over $k$, with Borel, Cartan, etc. subgroups denoted by $B_H,T_H$, etc. Let $\sA$ be a $H$-category with a compatible t-structure, $\sB$ be a category with a t-structure, and $g$ be a functor $\sA\rightarrow \sB$. If $g_{\sG}|_{N_H}$ is left t-exact up to shift for any compact $\sG\in D(H/N_H)$, then $g$ itself is left t-exact up to shift.
\end{lemma}

\begin{proof}
We claim that we have a commutative diagram
\begin{equation*}
\begin{tikzcd}
(\sA^{H,w}\underset{\QCoh(\ft^*_{\fh}//W_H)}{\otimes} \QCoh(\ft_{\fh}^{*,\circ}))^{\geq 0} \arrow[d,"\tau^{\geq 0}\circ\on{Loc}"] & \\
(\sA^{N_H})^{T_H,w}\underset{\QCoh(\ft^*_{\fh})}{\otimes} \QCoh(\ft_{\fh}^{*,\circ}) \arrow[d,"\Gamma"] \arrow[r] & (\sA^{N_H})^{T_H,w} \arrow[d]\\
\sA^{H,w}\underset{\QCoh(\ft^*_{\fh}//W_H)}{\otimes} \QCoh(\ft_{\fh}^{*,\circ}) \arrow[r] & \sA\arrow[d,"g"] \\
& \sB
\end{tikzcd}
\end{equation*}
where the functor from $(\sA^{N_H})^{T_H,w}$ to $\sA$ is given by $x\mapsto(\sD_{H/N_H}\overset{N_H}{\star}x)^{T_H}$. Indeed, to check the commutativity of the middle square it suffices to consider the universal case of $\sA\cong D(H)$, where the commutativity follows by unwinding the definition of $\Gamma$.

Note that we have $\Gamma\circ\tau^{\geq 0}\circ\on{Loc}\cong\tau^{\geq 0}\circ\Gamma\circ\on{Loc}\cong\tau^{\geq 0}\cong\on{id}$ on $(\sA^{H,w}\otimes_{\QCoh(\ft^*_{\fh}//W_H)}\QCoh(\ft_{\fh}^{*,\circ}))^{\geq 0}$. In particular, to show left t-exactness up to shift of the composition
\[
\sA^{H,w}\underset{\QCoh(\ft^*_{\fh}//W_H)}{\otimes} \QCoh(\ft_{\fh}^{*,\circ})\rightarrow \sA\rightarrow \sB,
\]
it suffices to show that the composed arrow from top left to bottom right lands in $\sB^+$. As $\tau^{\geq 0}\circ\on{Loc}$ is left t-exact by definition and $(\sA^{N_H})^{T_H,w}\otimes_{\QCoh(\ft^*_{\fh})}\QCoh(\ft_{\fh}^{*,\circ}) \rightarrow(\sA^{N_H})^{T_H,w}$ is t-exact by Lemma \ref{flattstructure}, we are reduced to showing that $(\sA^{N_H})^{T_H,w}\rightarrow \sB$ is left t-exact up to shift.

This functor is given by $x\mapsto g((\sD_{H/N_H}\overset{N_H}{\star}x)^{T_H})$. As $T_H$ is a torus, this is a direct summand of the functor $x\mapsto g(\sD_{H/N_H}\overset{N_H}{\star} x)$, which is left t-exact up to shift by assumption.

Next, recall that Lemma \ref{flattstructure} gives us a t-exact adjoint pair
\[
p^*:\sA^{H,w}\rightleftarrows \sA^{H,w}\underset{\QCoh(\ft^*_{\fh}//W_H)}{\otimes} \QCoh(\ft_{\fh}^{*,\circ}):p_*.
\]
Let us use this to show that $\sA^{H,w}\rightarrow \sB$ is left t-exact up to shift.

The Grothendieck-Cousin complex gives a finite resolution
\[
\sO_{\ft^*_{\fh}//W_H}\rightarrow\displaystyle\bigoplus_{x\in\ft^*_{\fh}//W_H\mid\on{codim}x=0}\sE_x\rightarrow\displaystyle\bigoplus_{x\in\ft^*_{\fh}//W_H\mid\on{codim}x=1}\sE_x\rightarrow\cdots.
\]
Here, for $x$ a (not necessarily closed) point of $\ft^*_{\fh}//W_H$, the sheaf $\sE_x$ is defined to be the colimit
\[
\on{colim}i_{k,*}i_k^!\sO_{\ft^*_{\fh}//W_H}[\on{codim} x]
\]
over the $k$th order neighborhood maps $i_k$.

Thus, for any object $\sF\in (\sA^{H,w})^{\geq 0}$, to show that $g(\sF)$ is bounded below, it suffices to show that the $g(\sE_x\otimes\sF)$ are bounded below uniformly in $x$. Using the above expression of $\sE_x$ as a filtered colimit, we see that it suffices to show the same for each $g(i_{k,*}i_k^!\sO_{\ft^*_{\fh//W_H}}[\on{codim} x]\otimes\sF)$. And as each coherent sheaf $i_{k,*}i_k^!\sO_{\ft^*_{\fh//W_H}}[\on{codim} x]$ is an iterated extension of skyscraper sheaves at $x$, it in fact suffices to show that the $g(k_x\otimes\sF)$ are uniformly bounded below.

Because $\ft_{\fh}^{*,\circ}\rightarrow\ft^*_{\fh}//W_H$ is faithfully flat, for any $x\in\ft^*_{\fh}//W_H$ there is a $y\in\ft_{\fh}^{*,\circ}$ mapping to $x$. Fix such $x$ and $y$. Then $k_x$ is a direct summand of the pushforward of $k_y$, and thus $k_x\otimes\sF$ is a direct summand of $p_*(k_y\otimes p^*\sF)$. Therefore, it suffices to show that the $g(p_*(k_y\otimes p^*\sF))$ are uniformly bounded below. Since we previously showed that $g\circ p_*$ is left t-exact up to shift, this follows from the $k_y\otimes p^*\sF$ being uniformly bounded below, which in turns follows from the existence of bounded length projective resolutions of $k_y$.

Finally, we can apply Lemma \ref{boundedinvariants} to deduce that $g$ is left t-exact up to shift, as desired.
\end{proof}

\subsubsection{Positive depth}

Assume that for all $s<r$ and $x\in X_*(T)\otimes\mathbb{R}$, each $f_{\sG}|_{K_{x,s+}}$ is left t-exact up to shift. We need to show, for some fixed such $x$, that $f|_{K_{x,r+}}$ is left t-exact up to shift. By replacing $f$ with $f_{\sG}$, our argument will also show the left t-exactness up to shift of $f_{\sG}|_{K_{x,r+}}$.

For any quasi-projective variety $V$ mapping to $(K_{x,r}/K_{x,r+})^*$, there is a functor
\[
f|_{K_{x,r+},V}:\sC^{K_{x,r+}}\underset{D((K_{x,r}/K_{x,r+})^*)}{\otimes} D(V)\rightarrow \sD.
\]
Assume that $U\subseteq V$ is an open subvariety with complement $Z$. Then, as in the proof of Lemma \ref{exactnessreductionstep}, the left t-exactness up to shift of $f|_{K_{x,r+},V}$ is reduced to the same property for $f|_{K_{x,r+},U}$ and $f|_{K_{x,r+},Z}$. Once again, this reduces the desired statement to showing that $f|_{K_{x,r+},O}$ is left t-exact up to shift for any $L_x$-orbit $O$ in $(K_{x,r}/K_{x,r+})^{*,us}$.

By Lemma \ref{geometricinput}, there is a point $p\in O$ and a $y\in X_*(T)\otimes\mathbb{R}$ such that $K_{x,r+}\subseteq K_{y,r}\subseteq K_{x,r}$ and $p\in(K_{y,r}/K_{x,r+})^{\perp}$. Then $f|_{K_{x,r+},p}$ can be identified with the composition
\[
\sC^{K_{x,r+},p}\rightarrow \sC^{K_{x,r+},(K_{y,r}/K_{x,r+})^{\perp}}\cong \sC^{K_{y,r}}\rightarrow \sD,
\]
which is left t-exact up to shift by Lemma \ref{exactsequence} and the inductive hypothesis.

Let $S_p$ be the stabilizer group of $p$ with respect to the action of $P_x/K_{x,r+}$ on $O$. By Lemma \ref{orbitstabilizerequivalence}, we have a t-exact equivalence
\[
\sC^{K_{x,r+},O}\cong \sC^{K_{x,r+},p}\underset{S_p}{\otimes} D(P_x/K_{x,r+}).
\]
Thus, it suffices to show that $\sC^{K_{x,r+},p}\otimes_{S_p}D(P_x/K_{x,r+})\rightarrow \sD$ is left t-exact up to shift.

First we show this for $\sC^{K_{x,r+},p}\otimes D(P_x/K_{x,r+})\rightarrow \sD$. Let $\sF$ be an object of $(\sC^{K_{x,r+},p}\otimes D(P_x/K_{x,r+}))^+$. By Lemma 11.2.4 of \cite{raskinbeiber}, $\sF$ lies in the subcategory generated under finite colimits and direct summands by the object
\[
(\on{id}_{\sC^{K_{x,r+},p}}\otimes\Gamma_{\on{IndCoh}}(P_x/K_{x,r+},-))(\sF)\boxtimes D_{P_x/K_{x,r+}}.
\]

It thus suffices to show that
\[
f_{D_{P_x/K_{x,r+}}}((\on{id}_{\sC^{K_{x,r+},p}}\otimes\Gamma_{\on{IndCoh}}(P_x/K_{x,r+},-))(\sF))\in \sD^+,
\]
or, from the inductive hypothesis, that
\[
(\on{id}_{\sC^{K_{x,r+},p}}\otimes\Gamma_{\on{IndCoh}}(P_x/K_{x,r+},-))(\sF)\in (\sC^{K_{x,r+},p})^+.
\]
This is true as $\on{id}_{\sC^{K_{x,r+},p}}\otimes\Gamma_{\on{IndCoh}}(P_x/K_{x,r+},-)$ is left t-exact, because $\Gamma_{\on{IndCoh}}(P_x/K_{x,r+},-)$ is left t-exact and admits a left adjoint.

Now let $\sG$ be an object in $(\sC^{K_{x,r+},O})_{\geq 0}$. We would like to show that $f(\sG)\in \sD^+$. To do so, we use a Grothendieck-Cousin argument as in the depth zero case. Recall that we have a resolution
\[
\sO_{O}\rightarrow\displaystyle\bigoplus_{y\in O\mid\on{codim}y=0}\sE_y\rightarrow\displaystyle\bigoplus_{y\in O\mid\on{codim}y=1}\sE_y\rightarrow\cdots.
\]
This time, we will use the interpretation of $\sE_y$ as the D-module of delta functions at $y$. Evidently, it suffices to show that the $f(\sG\otimes\sE_y)$ are bounded below uniformly in $y$.

Recall that we have adjoint functors
\[
\on{Oblv}:\sC^{K_{x,r+},p}\underset{S_p}{\otimes} D(P_x/K_{x,r+})\rightleftarrows \sC^{K_{x,r+},p}\otimes D(P_x/K_{x,r+}):\on{Av}_*,
\]
with $\on{Oblv}$ t-exact. Pick some point $z\in P_x/K_{x,r+}$ with $z\cdot p=y$. Then $\sG\otimes\sE_y$ is a direct summand of $\on{Av}_*(\on{Oblv}\sG\otimes\sE_z)$, and so $f(\sG\otimes\sE_y)$ is a direct summand of $f(\on{Av}_*(\on{Oblv}\sG\otimes\sE_z))$. As we previously showed that $f\circ\on{Av}_*$ is left t-exact up to shift, we are reduced to showing that the $\on{Oblv}\sG\otimes\sE_z$ are bounded below.

For this, we invoke Lemma 11.2.4 of \cite{raskinbeiber} again. We see that $\on{Oblv}\sG$ lies in the subcategory generated under finite colimits and direct summands by the object
\[
(\on{id}_{\sC^{K_{x,r+},p}}\otimes\Gamma_{\on{IndCoh}}(P_x/K_{x,r+},-))(\on{Oblv}\sG)\boxtimes D_{P_x/K_{x,r+}},
\]
 so $\on{Oblv}\sG\otimes\sE_z$ lies in the subcategory generated under finite colimits and direct summands by the object
\[
(\on{id}_{\sC^{K_{x,r+},p}}\otimes\Gamma_{\on{IndCoh}}(P_x/K_{x,r+},-))(\on{Oblv}\sG)\boxtimes (D_{P_x/K_{x,r+}}\otimes\sE_z).
\]
But as $(D_{P_x/K_{x,r+}}\otimes\sE_z)$ is uniformly bounded below, so is this object, and thus so is $\on{Oblv}\sG\otimes\sE_z$, as desired.

\end{proof}

\subsection{Proof of Lemma \ref{boundedness}}

Lemma \ref{boundedness} now follows.

\begin{proof}
We need to check the conditions of Lemma \ref{boundednessreductionstep}. The statement on $(K_{x,r+},\circ)$-invariants follows immediately from Theorem \ref{grdepthzero}. And the statement on $I_0$-invariants is just Lemma \ref{i0boundedness}.
\end{proof}

\appendix

\section{Tensor products and t-structures}

As in the rest of this paper, all categories will be assumed presentable and all t-structures will be assumed accessible, right complete, and compatible with filtered colimits. We will often implicitly invoke 1.4.4.11 of \cite{HA}, which states that the data of an accessible t-structure on a presentable category $\sC$ is equivalent to the data of a full subcategory $\sC^{\leq 0}$ which is presentable, closed under small colimits, and closed under extensions.\footnote{To avoid confusion, we note that we use the cohomological degree convention, while \cite{HA} uses the homological convention. In particular, the category we call $\sC^{\leq 0}$ is instead referred to as $\sC_{\geq 0}$ in \emph{loc. cit}.} Furthermore, for any small collection of objects of $\sC$, the category $\sC^{\leq 0}$ generated by them under colimits and extensions is presentable (and thus defines a t-structure.)

\subsection{Tensor product}

Let $\sC$ and $\sD$ be two categories equipped with t-structures. Then we can equip $\sC\otimes \sD$ with a t-structure by defining $(\sC\otimes \sD)^{\leq 0}$ to be generated (under colimits and extensions) by objects of the form $c\boxtimes d$, with $\sC$ in $\sC^{\leq 0}$ and $d$ in $\sD^{\leq 0}$. This t-structure will be right complete and compatible with filtered colimits by Theorem C.4.2.1 in \cite{SAG}.

\begin{lemma}\label{rightexacttensor}
Let $\sC$ and $\sC'$ be categories equipped with t-structures and let $f:\sC\rightarrow \sC'$ be a right t-exact functor. Then, for any category $\sD$ with a t-structure,
\[
f\otimes \on{id}_{\sD}:\sC\otimes \sD\rightarrow \sC'\otimes \sD
\]
is right t-exact.
\end{lemma}

\begin{proof}
We need to show that $f\otimes \sD$ sends $(\sC\otimes \sD)^{\leq 0}$ to $(\sC'\otimes \sD)^{\leq 0}$. It suffices to check that, for $c$ in $\sC^{\leq 0}$ and $d$ in $\sD^{\leq 0}$,
\[
(f\otimes \on{id}_{\sD})(c\boxtimes d)\cong f(c)\boxtimes d\in(\sC'\otimes \sD)^{\leq 0}.
\]
But this follows from the right t-exactness of $f$.
\end{proof}

The stability of left t-exactness under tensor product is more subtle. We start with a simple observation.

\begin{lemma}\label{adjointtensor}
Let $\sC$ and $\sC'$ be categories equipped with t-structures and let $f:\sC\rightarrow \sC'$ be a left t-exact functor with a left adjoint $f^L$. Then, for any category $\sD$ with a t-structure,
\[
f\otimes \on{id}_{\sD}:\sC\otimes \sD\rightarrow \sC'\otimes \sD
\]
is left t-exact.
\end{lemma}

\begin{proof}
Recall that left t-exactness of $f$ is then equivalent to right t-exactness of $f^L$. By Lemma \ref{rightexacttensor}, $f^L\otimes\on{id}_{\sD}$ is right t-exact, so its right adjoint $f\otimes\on{id}_{\sD}$ must be left t-exact, as desired.
\end{proof}

We will prove two other criteria for a tensor product to be left t-exact. Both revolve around the following notion.

\begin{definition}
Let $\sD$ be a category equipped with a t-structure. We say that the t-structure is compactly generated if $\sD^{\leq 0}$ is generated under colimits by $\sD^{\leq 0}\cap \sD^c$.
\end{definition}

\begin{example}
Let $X$ be a quasiseparated scheme of finite type over a field. Then the natural t-structures on $\on{QCoh}(X)$, $\on{IndCoh}(X)$, and $D(X)$ are all compactly generated. For $\IndCoh(X)$, this is because $\on{Coh}(X)$ is closed under truncations, and the case of $D(X)$ follows as $\on{ind}(\IndCoh(X)_{\leq 0})$ generates $D(X)$. For $\QCoh(X)$, this follows from a slight modification of the arguments in \cite{ThoTro}.
\end{example}

For $\sD$ a DG category and $\sF \in \sD^c$ a compact object,
we let $\bD \sF:\sD \to \on{Vect}$ denote the induced functor
$\Hom(\sF,-)$.

\begin{lemma}\label{l:hom-left-pres}

Let $\sC$ and $\sD$ be DG categories with t-structures. Suppose 
$\sF \in \sD^c \cap \sD^{\leq 0}$. 
Then the induced functor:
\[
\on{id}_{\sC} \otimes \bD \sF : \sC \otimes \sD \to \sC \otimes \on{Vect} = 
\sC
\]

\noindent is left t-exact.

\end{lemma}

\begin{proof}

Note that the functors $- \otimes \sF:\on{Vect} \rightleftarrows
\sD: \bD \sF$ are adjoint in the symmetric monoidal 
2-category $\DGCat$. As $-\otimes\sF$ is right t-exact, we see that $\bD \sF$ is left t-exact. The result now follows from Lemma \ref{adjointtensor}.

\end{proof}

Our first left t-exactness criterion appears, e.g., as Lemma B.6.2. of \cite{Whit}, but we reproduce it here for the reader's convenience.

\begin{proposition}\label{cgtensor1}
Let $\sC$ and $\sC'$ be categories equipped with t-structures and let $f:\sC\rightarrow \sC'$ be a left t-exact functor. Then, for any category $\sD$ with a compactly generated t-structure,
\[
f\otimes \on{id}_{\sD}:\sC\otimes \sD\rightarrow \sC'\otimes \sD
\]
is left t-exact.
\end{proposition}

\begin{proof}
Let $x$ be an object of $(\sC\otimes \sD)^{\geq 0}$. We would like to show that $(f\otimes\on{id}_{\sD})(x)\in(\sC'\otimes \sD)^{\geq 0}$, or equivalently, that
\[
\on{Hom}_{\sC'\otimes \sD}(c\otimes d[1], (f\otimes\on{id}_{\sD})(x))\cong 0
\]
for all $c\in \sC^{\leq 0}, d\in \sD^{\leq 0}$.

Because the t-structure on $\sD$ is compactly generated, we can assume that $d$ is compact. We compute:
\begin{align*}
&\on{Hom}_{\sC'\otimes \sD}(c\otimes d[1], (f\otimes\on{id}_{\sD})(x))\\
\cong&\on{Hom}_{\sC'}(c[1],(\on{id}_{\sC} \otimes \bD \sF)((f\otimes\on{id}_{\sD})(x)))\\
\cong&\on{Hom}_{\sC'}(c[1],f((\on{id}_{\sC} \otimes \bD \sF)(x))).\\
\end{align*}

Combining Lemma \ref{l:hom-left-pres} with the left t-exactness of $f$, we see that this Hom space is trivial, as desired.

\end{proof}

The other left t-exactness criterion appears below as Proposition \ref{cgtensor2}. We first need a preliminary lemma.

\begin{lemma}\label{l:pro}

Suppose $\sD$ is a DG category with a compactly generated
t-structure. 

Let $\sD^{\vee,\prime} \subset \sD^{\vee}$ denote the 
subcategory generated under colimits by objects of the form
$\bD \sF$ for $\sF \in \sD^c \cap \sD^{\leq 0}$.

Then $\lambda \in \sD^{\vee}$ lies in 
$\sD^{\vee,\prime}$ if and only if the functor $\lambda:\sD \to \on{Vect}$ is 
left t-exact. 

\end{lemma}

\begin{proof}

Recall that the functor $\sD^{c,op} \xar{\sF \mapsto \bD \sF} \sD^{\vee}$ 
extends to an equivalence $\on{Ind}(\sD^{c,op}) \cong \sD^{\vee}$.
We wish to show that our left t-exact $\lambda$ lies in 
the subcategory $\on{Ind}(\sD^{c,op} \cap \sD^{\leq 0,op}) \subset 
\on{Ind}(\sD^{c,op}) \cong \sD^{\vee}$.
Concretely, this amounts to showing that for any 
$\sF \in \sD^c$ and any map $\alpha:\bD\sF \to \lambda \in \sD^{\vee}$,
there exists $\sF_0 \in \sD^c \cap \sD^{\leq 0}$ and a map
$\beta:\sF_0 \to \sF$ such that the map $\alpha$ factors
as
\[
\bD \sF \xar{\bD \beta} \bD \sF_0 \to \lambda.
\]

For this, recall the map $\bD \sF \to \lambda$ amounts
to a point $\alpha \in \Omega^{\infty} \lambda(\sF)$.
As $\lambda$ is left $t$-exact, the map
$\Omega^{\infty} \lambda(\tau^{\leq 0} \sF) \to \Omega^{\infty}
\lambda(\sF)$ is an isomorphism. Now write
$\tau^{\leq 0} \sF$ as a filtered colimit $\on{colim}_i \sF_i$
with $\sF_i \in \sD^c \cap \sD^{\leq 0}$;
we may do so because the t-structure is compactly generated.
Then the point $\alpha$ comes from a point
$\alpha_i \in \Omega^{\infty} \lambda(\sF_i)$ for some $i$;
taking $\sF_0 \coloneqq \sF_i$ yields the claim.

\end{proof}

\begin{corollary}\label{c:lambda-left-pres}

Suppose $\sD$ is a DG category with a compactly generated
t-structure and $\sC$ is a DG category with a t-structure.
Suppose $\lambda:\sD \to \on{Vect}$ is left t-exact.

Then the induced functor
\[
(\on{id}_{\sC} \otimes \lambda): \sC \otimes \sD \to \sC \otimes \on{Vect} =
\sC
\]

\noindent is left t-exact.

\end{corollary}

\begin{proof}

By Lemma \ref{l:hom-left-pres}, the result is true in the special case
$\lambda = \bD \sF$ with $\sF \in \sD^c \cap \sD^{\leq 0}$.
By Lemma \ref{l:pro}, any left $t$-exact $\lambda$ can 
be written as a filtered colimit of functors of this form.
As our t-structures are assumed compatible with filtered
colimits, left t-exact functors are preserved under filtered colimits.
This yields the claim.

\end{proof}

\begin{proposition}\label{cgtensor2}

Suppose $\sD_1$ and $\sD_2$ are DG categories with compactly 
generated t-structures and let $\sC$ be a DG category with a t-structure.

Suppose $F:\sD_1 \to \sD_2$ is left t-exact.
Then the induced functor:
\[
\on{id}_{\sC} \otimes F: \sC \otimes \sD_1 \to \sC \otimes \sD_2
\]

\noindent is left t-exact. 

\end{proposition}

\begin{proof}

Suppose $\sG \in (\sC \otimes \sD_1)^{>0}$ is given. We wish
to show that 
\[
(\on{id}_{\sC} \otimes F)(\sG) \in (\sC \otimes \sD_2)^{>0}.
\]

\noindent Equivalently, as the t-structure on $\sD_2$ is compactly
generated, it suffices to show that for 
$\sF \in \sD_2^c \cap \sD_2^{\leq 0}$ and $\sH \in \sC^{\leq 0}$,
we have
\[
\Hom_{\sC \otimes \sD_2}(
\sH \boxtimes \sF,
(\on{id}_{\sC} \otimes F)(\sG)) = 0.
\]

\noindent As in the proof of Lemma \ref{l:hom-left-pres}, we may rewrite
the left hand side as
\[
\Hom_{\sC}(\sH,
(\on{id}_{\sC} \otimes \bD \sF) \circ (\on{id}_{\sC} \otimes F)(\sG)) =
\Hom_{\sC}(\sH,
F^{\vee}(\bD\sF)(\sG)).
\]

\noindent Observe that $F^{\vee}(\bD \sF) = \bD \sF \circ F:\sD_1 \to \on{Vect}$ is
left t-exact as $F$ is left t-exact and $\sF$ is connective.
Therefore, $F^{\vee}(\bD\sF)(\sG) \in \sC^{>0}$ by assumption on 
$\sG$. Therefore, as $\sH \in \sC^{\leq 0}$, the above
term vanishes as desired. 

\end{proof}

\subsection{Compatible t-structures}

Let $H$ be an affine algebraic group (in particular, of finite type). Following Appendix B of \cite{Whit} and Sections 10.9-13 of \cite{semiinf}, we will introduce a notion of compatibility between a t-structure on $\sC$ and a $H$-action (strong or weak) on $\sC$. No originality is claimed for any of the results in this section.

The key lemma is the following:

\begin{lemma}[\cite{Whit}]\label{compatiblelemma}
Let $\sC$ be a category with a weak action of $H$. Assume we have a t-structure on $\sC$. Then the following conditions are equivalent:
\begin{enumerate}
\item The functor $\on{Oblv}\circ\on{Av}_*^w:\sC\rightarrow \sC$ is t-exact.
\item The functor $\on{coact}:\sC\rightarrow\QCoh(H)\otimes \sC$ is t-exact.
\item The functor $\on{act}:\QCoh(H)\otimes \sC\rightarrow \sC$ is t-exact.
\item The category $\sC^{H,w}$ admits a t-structure for which $\on{Oblv}$ and $\on{Av}_*^w$ are t-exact.
\item The $\QCoh(H)$-linear equivalence
\[
\QCoh(H)\otimes \sC\rightarrow\QCoh(H)\otimes \sC
\]
induced by $\on{coact}$ is t-exact.
\end{enumerate}
\end{lemma}

\begin{proof}
This is Lemma B.3.1 of \cite{Whit}. Let us reproduce the proof.

First, note that we have an adjoint pair
\[
p_2^*: \sC\rightleftarrows\QCoh(H)\otimes \sC :p_{2,*}
\]
induced by the pullback and pushforward maps between $\on{Vect}\cong\QCoh(\on{Spec} k)$ and $\QCoh(H)$. By Lemma \ref{rightexacttensor}, both $p_2^*$ and $p_{2,*}$ are right t-exact. The adjunction then implies that $p_{2,*}$ is also left t-exact, hence t-exact. Noting that $p_{2,*}$ is conservative and that $p_{2,*}\circ p_2^*\cong\sO_H\otimes-$ is t-exact, we see that $p_2^*$ is t-exact as well.

To prove conditions (1) through (4) are equivalent, we use natural isomorphisms
\[
p_{2,*}\circ\on{coact}\cong\on{act}\circ p_2^*\cong\on{Oblv}\circ\on{Av}_*^w.
\]
To see these isomorphisms, note that by a standard tensor product argument, it suffices to check the universal case of $\sC\cong\QCoh(H)$, where the desired isomorphisms follow as from the base change formula.

$(1)\iff(2):$ As $p_{2,*}$ is t-exact and conservative, we see that $\on{coact}$ is t-exact if and only if $p_{2,*}\circ\on{coact}\cong\on{Oblv}\circ\on{Av}_*^w$ is t-exact, as desired.

$(2)\implies(3):$ As $\on{act}$ is right adjoint to $\on{coact}$, we see that $\on{act}$ is left t-exact. To show right t-exactness, it suffices to note that $p_2^*$, viewed as a functor from $\sC^{\leq 0}$ to $(\QCoh(H)\otimes \sC)^{\leq 0}$, generates under colimits.

$(3)\implies(1):$ Clear.

$(3)\iff(5):$ The equivalence of (5) intertwines $p_2^*$ and $\on{coact}$, so $(5)$ implies $(3)$. On the other hand, if $\on{coact}$ is t-exact, it suffices to show that $\on{coact}$, restricted to a functor $\sC^{\leq 0}\rightarrow (\QCoh(H)\otimes \sC)^{\leq 0}$, generates under colimits. By adjunction, this would follow if we knew $\on{act}$ to be conservative and t-exact. The t-exactness follows from $(2)\implies(3)$, and to check conservativeness it suffices to consider the universal case of $\sC\cong\QCoh(H)$, where it is obvious.

$(4)\implies(1):$ Clear.

$(1)+(2)\implies(4):$ This is the hardest implication to prove. Recall that $\sC^{H,w}$ is the totalization of the cosimplicial category
\[
\sC\rightrightarrows \QCoh(H)\otimes \sC \rightthreearrow \QCoh(H)\otimes\QCoh(H)\otimes \sC\cdots.
\]
To construct a t-structure on $\sC^{H,w}$, it suffices to prove that all of the transition maps are t-exact. Consider a face map $f:\sC\otimes(\QCoh(H))^n\rightarrow \sC\otimes(\QCoh(H))^{n+1}$. We split into cases depending on the value of $n$.

For $n=0$ there are two face maps, $p_2^*$ and $\on{coact}$. By our assumptions, both are t-exact.

Next, for $n=1$, there are three face maps, $p_2^*\otimes\on{id}_{\QCoh(H)}, \on{coact}\otimes\on{id}_{\QCoh(H)}$, and $\on{id}_{\sC}\otimes\Delta_*$, where $\Delta_*$ is pushforward along the diagonal map $\Delta:H\rightarrow H\times H$. The first two maps are t-exact by Proposition \ref{cgtensor1}, and $\on{id}_{\sC}\otimes\Delta_*$ is t-exact by Lemma \ref{adjointtensor}.

Finally, if $n>1$, then $f$ is of the form $g\otimes\on{id}_{\QCoh(H)}$ for some face map $g:\sC\otimes(\QCoh(H))^{n-1}\rightarrow \sC\otimes(\QCoh(H))^n$. Using Proposition \ref{cgtensor1} again, this case follows by induction.

By construction, $\on{Oblv}$ is t-exact with respect to this t-structure. As $\on{Oblv}$ is conservative and $\on{Oblv}\circ\on{Av}_*^w$ is t-exact, we see that $\on{Av}_*^w$ is t-exact, as desired.
\end{proof}

\begin{definition}
Let $\sC$ be a category with a weak action of $H$. A t-structure on $\sC$ is said to be compatible with the $H$ action if it satisfies the equivalent conditions of Lemma \ref{compatiblelemma}. If $\sC$ is instead endowed with a strong $H$-action, we say that a t-structure on $\sC$ is compatible with the $H$ action if it is compatible with the underlying weak $H$-action.
\end{definition}

In the case of a strong action on $\sC$, note that the (strong) coaction functor $\sC\rightarrow \sC\otimes D(H)$ is still t-exact, up to shift. Indeed, the composition $\sC\rightarrow \sC\otimes D(H)\rightarrow \sC\otimes\QCoh(H)$ is, up to shift, the weak coaction functor, which is t-exact by assumption. On the other hand, as $D(H)\rightarrow\QCoh(H)$ is exact, conservative, and has a right adjoint, the same properties hold for $\sC\otimes D(H)\rightarrow \sC\otimes\QCoh(H)$, which together imply that $\sC\rightarrow \sC\otimes D(H)$ is t-exact, as desired. This allows us to show:

\begin{lemma}\label{compatibleinvariants}
Let $\sC$ be a category with a strong action of $H$ and a compatible t-structure. Then the invariant category $\sC^H$ has a natural t-structure with $\on{Oblv}:\sC^H\rightarrow \sC$ t-exact.
\end{lemma}

\begin{proof}
This follows by the same logic as in the proof of $(1)+(2)\implies (4)$ in Lemma \ref{compatiblelemma}. The only point requiring justification is the left t-exactness of $\on{id}_{\sC}\otimes\Delta_*: \sC\otimes D(H)\rightarrow\sC\otimes D(H\times H)$, which is true by Proposition \ref{cgtensor2}.
\end{proof}

As a consequence of Lemma \ref{compatibleinvariants}, we can construct t-structures on $H$-tensor products. 

\begin{lemma}\label{compatibletensor}
Let $\sC$ be a category with a right $H$-action and let $\sD$ be a category with a left $H$-action. Assume both $\sC$ and $\sD$ are equipped with compatible t-structures. Then $\sC\otimes_H \sD$ has a natural t-structure, and there is a natural t-exact functor $\sC\otimes_H \sD\rightarrow \sC\otimes \sD$.
\end{lemma}

\begin{proof}
Composing with the inversion map $D(H)\rightarrow D(H)$, we can make $\sC$ into a left $H$-category instead. Then $\sC\otimes \sD$ has a left $H\times H$ action. We claim that there is a canonical equivalence $\sC\otimes_H \sD\cong (\sC\otimes \sD)^H$, where the $H$-invariants on the RHS is taken with respect to the diagonal copy of $H$ inside $H\times H$. Indeed, there is an obvious such equivalence for $\sC\cong \sD\cong D(H)$, and the general case follows by tensoring.

Thus to equip $\sC\otimes_H \sD$ with a t-structure satisfying the desired properties, it suffices to show that the natural t-structure on $\sC\otimes \sD$ is compatible with the diagonal $H$ action. In fact, it is compatible with the entire $H\times H$ action. To see this, we need to show that the functor
\[
\QCoh(H)\otimes \sC\otimes\QCoh(H)\otimes \sD\rightarrow \sC\otimes \sD
\]
is t-exact. By assumption, the functors $\QCoh(H)\otimes \sC\rightarrow \sC$ and $\QCoh(H)\otimes \sD\rightarrow \sD$ are t-exact. As they are also right adjoints, Lemma \ref{rightexacttensor} and Lemma \ref{adjointtensor} tell us that their tensor product is t-exact, as desired.
\end{proof}

Finally, following Sections 10.9-13 of \cite{semiinf}, let us say something about the case where $H$ is, rather than a finite type algebraic group, a Tate group indscheme with prounipotent tail (e.g., $G((t))$.) As we will make only cursory use of this material, our treatment will only be a sketch. In this case, compatibility with a weak $H$-action on $\sC$ is defined to mean t-exactness of the equivalence
\[
\IndCoh^*(H)\otimes \sC\rightarrow\IndCoh^*(H)\otimes \sC
\]
intertwining $\on{act}$ and $p_{2,*}$.

If $\sC$ comes with a strong $H$-action, we impose one more requirement, namely, that for all (or equivalently, one) prounipotent compact subgroups $K\subseteq H$, $\sC^K$ is closed under truncations. This is equivalent to the existence of a t-structure on $\sC^K$ such that $\sC^K\rightarrow \sC$ is t-exact.

\subsection{Case of additive groups}

Now take $H$ to be a vector space $V$, treated as an algebraic group with group structure given by addition. By the Fourier-Deligne transform for $D$-modules, there is a t-exact equivalence $D(V)\cong D(V^*)$ intertwining the convolution monoidal structure on the left with the $\otimes^!$ monoidal structure on the right. In particular, if $X$ is a variety mapping to $V^*$, there is a natural $H$-action on $D(X)$.

We will be interested in categories of the form $\sC\otimes_A D(X)$, for $\sC$ a $H$-category. Assume that $\sC$ comes with a compatible t-structure. Then Lemma \ref{compatibletensor} gives a t-structure on $\sC\otimes_H D(X)$. Furthermore, the proof of the lemma shows that this t-structure is compatible with the $H$-action.

The main result of this section is the following.

\begin{lemma}\label{exactsequence}
Assume that $H$ corresponds to a vector space $V$, and let $\sC$ be a $H$-category with a t-structure compatible with the $H$-action. Let $X$ be a quasi-projective variety mapping to $V^*$.
\begin{enumerate}
\item If $i:Z\rightarrow X$ is a closed embedding, then $\on{id}_{\sC}\otimes_{H}i_*:\sC\otimes_H D(Z)\rightarrow \sC\otimes_H D(X)$ is t-exact.
\item If $j:U\rightarrow X$ is an open immersion, then $\on{id}_{\sC}\otimes_H j^*:\sC\otimes_H D(X)\rightarrow \sC\otimes_H D(U)$ is t-exact.
\end{enumerate}
\end{lemma}

\begin{proof}
Let us first treat the case of closed embeddings. We have a commutative diagram
\begin{equation*}
\begin{tikzcd}
\sC\underset{H}{\otimes} D(Z)\arrow[r] \arrow[d] & \sC\underset{H}{\otimes} D(X)\arrow[d] \\
\sC\otimes D(Z) \arrow[r] & \sC\otimes D(X)
\end{tikzcd}
\end{equation*}
with vertical arrows conservative and t-exact. Thus, to show that the top arrow is t-exact, it suffices to show that the bottom arrow $\on{id}_{\sC}\otimes i_*$ is t-exact. But this follows from Lemma \ref{rightexacttensor} and Proposition \ref{cgtensor2}.

Now we consider the case of open immersions. Again, we are reduced to showing that $\on{id}_{\sC}\otimes j^*$ is t-exact. And again, right t-exactness is automatic by Lemma \ref{rightexacttensor} and left t-exactness holds because of Proposition \ref{cgtensor2}.
\end{proof}

\printbibliography

\end{document}